\documentclass[10pt,leqno,amscd,amssymb,verbatim, url]{amsart}
\usepackage{amsfonts,amssymb,graphicx}
\usepackage{amsmath,amscd}

\usepackage{hyperref}

\newcommand{\nc}{\newcommand}
\nc\on{\operatorname}
\usepackage{color}

\nc{\red}[1]{\textcolor{red}{#1}}
\nc{\blue}[1]{\textcolor{blue}{#1}}
\nc{\green}[1]{\textcolor{green}{#1}}
\nc{\dn}[1]{\red{From DN: #1}}
\nc{\zl}[1]{\blue{ From ZL: #1}}

\newtheorem{theorem}{Theorem}[subsection]
\newtheorem{proposition}[theorem]{Proposition}

\newtheorem{corollary}[theorem]{Corollary}
\newtheorem{lemma}[theorem]{Lemma}

\theoremstyle{definition}

\newtheorem{question}[theorem]{Question}

\numberwithin{equation}{subsection}

\newcommand{\delete}[1]{}

\newcommand{\BC}{{\mathbb C}}

\newcommand{\HH}{\operatorname{H}}

\newcommand{\Ext}{\operatorname{Ext}}
\renewcommand{\hom}{\operatorname{Hom}}

\newcommand{\Lie}{\operatorname{Lie}}
\renewcommand{\ker}{\operatorname{Ker}}

\newcommand{\BU}{{\mathbb U}}
\newcommand{\bfk}{{\mathsf{k}}}

\usepackage{algorithmic}

\newcommand{\blist}{\begin{list}{\rom{(\roman{enumi})}}{\setlength{\leftmargin}{0em}
\setlength{\itemindent}{7ex}
\setlength{\labelsep}{2ex}\setlength{\listparindent}{\parindent}
\usecounter{enumi}}}
\newcommand{\elist}{\end{list}}

\begin{document}

\title[Rings of regular functions]{Realizing rings of regular functions via the cohomology of quantum groups}



\author[Zongzhu Lin]{Zongzhu Lin}
\address{Department of Mathematics\\
Kansas State University \\
Manhattan, KS 66506}
\email{zlin@math.ksu.edu}


\author[Daniel K. Nakano]{Daniel K. Nakano}
\address{Department of Mathematics \\
University of Georgia\\
Athens, GA 30602} 
\email{nakano@math.uga.edu}
\thanks{Research of the second author was supported in part by NSF
grant  DMS-2101941}

\subjclass[2010]{Primary 20G42, 20G10 ; Secondary 17B56}

\date{}

\dedicatory{In memory of Georgia M. Benkart and Brian J. Parshall}

\begin{abstract} Let $G$ be a complex reductive group and $P$ be a parabolic subgroup of $G$. 
In this paper the authors address questions involving the realization of the $G$-module of the global sections 
of the (twisted) cotangent bundle over the flag variety $G/P$ via the cohomology of the small quantum group. 
Our main results generalize the important computation of the cohomology ring for the small 
quantum group by Ginzburg and Kumar \cite{GK}, and provides a generalization of well-known  
calculations by Kumar, Lauritzen, and Thomsen \cite{KLT} to the quantum case and the parabolic setting. As an application we 
answer the question (first posed by Friedlander and Parshall for Frobenius kernels \cite[(3.2)]{FP}) about the realization of 
coordinate rings of Richardson orbit closures for complex semisimple groups via quantum group cohomology. Formulas will be provided 
which relate the multiplicities of simple $G$-modules in the global sections with the dimensions of extension groups over the large quantum group. 
\end{abstract}

\maketitle


\section{Introduction} 

\subsection{} Let $G$ be a complex simple algebraic group, 
${\mathfrak g}$ be its Lie algebra and $\Phi$ be the associated root system. 
Fix $\zeta$ an ${\ell}$th root of unity in ${\mathbb C}$. Using divided powers, Lusztig defined 
a ${\mathbb Z}[q,q^{-1}]$-form, from which one  can specialize $q$ to an $\ell$-th primitive root $\zeta\in \mathbb{C}$ of 1 to get 
 $U_{\zeta}({\mathfrak g})$.  We will call  $U_{\zeta}({\mathfrak g})$ the large quantum group. 
The small quantum group $u_{\zeta}({\mathfrak g})$ is a finite-dimensional normal  Hopf subalgebra of 
$U_{\zeta}({\mathfrak g})$ whose quotient $U_{\zeta}({\mathfrak g})/\!\!/u_{\zeta}({\mathfrak g})$ identifies with 
the ordinary universal enveloping algebra ${\mathcal U}({\mathfrak g})$. This basic fact implies that 
for any $U_{\zeta}({\mathfrak g})$-module, the cohomology group 
$\HH^{\bullet}({\mathfrak u}_{\zeta}({\mathfrak g}),M)$ is a rational $G$-module. 
When $M\cong {\mathbb C}$ (i.e., the trivial module), 
Ginzburg and Kumar \cite{GK} demonstrated that for $\ell>h$ ($h$ is the Coxeter number for $\Phi$) 
the odd degree cohomology vanishes and $\HH^{2\bullet}({\mathfrak u}_{\zeta}({\mathfrak g}),{\mathbb C})
\cong {\mathbb C}[{\mathcal N}]$ where ${\mathcal N}$ is the variety 
of nilpotent elements of ${\mathfrak g}$. Arkhipov, Bezrukavnikov and Ginzburg \cite{ABG} have used 
this computation as a starting point to show that there are beautiful connections between representations 
for quantum groups and the complex algebraic geometry of ${\mathcal N}$. As an application of 
their work, they provided a proof of Lusztig character formula for simple modules in the 
quantum group setting. 

In \cite{BNPP}, Bendel, Nakano, Parshall and Pillen  have also exploited the 
powerful tools available in complex algebraic geometry  to compute the cohomology 
ring $\HH^{\bullet}({\mathfrak u}_{\zeta}({\mathfrak g}),{\mathbb C})$ 
when $\ell \leq h$. Their computation verified that the cohomology ring is 
finitely generated and allowed them to develop a theory of support varieties. Furthermore, 
they computed the support varieties of quantum Weyl modules in the case when $(p,\ell)=1$ for 
every bad prime $p$ for $\Phi$ which proves a quantum version of a conjecture of Jantzen 
(cf. \cite[(6.2.1) Theorem]{NPV}). 

\subsection{} Let $G_k$ be a reductive algebraic group over an algebraically closed field $k$ of positive characteristic $p>0$ with a rational structure over the prime field $ \mathbb{F}_p$. In 1986, Friedlander and Parshall \cite[(3.2)]{FP} posed several questions about 
(i) the realization of orbit closures via support varieties for finite-dimensional 
rational $G_k$-modules and (ii) the realization of the coordinate algebras of the closures of nilpotent $G_k$-orbits in the Lie algebra $ \mathfrak{g}_{k}=\Lie(G_k)$ 
via the cohomology of the first Frobenius kernel of the reductive 
group with coefficients in a suitable algebra. In the quantum group setting, 
the analog of (i) was verified by Bezrukavnikov (\cite {Be}) by computing 
the support varieties of tilting modules. The quantum analog of (ii) can be stated as follows: 
\begin{question} \label{realizequestion1} 
Let $\ell>h$ and $J\subseteq \Delta$ with Levi decomposition 
of ${\mathfrak g}={\mathfrak u}_{J}\oplus {\mathfrak l}_{J}\oplus {\mathfrak u}_{J}^{+}$. 
Does there exist a $U_{\zeta}({\mathfrak g})$-algebra $A$ such that 
$\HH^{\text{odd}}(u_{\zeta}({\mathfrak g}),A)=0$ and 
$\HH^{2\bullet}(u_{\zeta}({\mathfrak g}),A)={\mathbb C}[G\cdot {\mathfrak u}_{J}]$?
\end{question} 
\noindent 
We will also be interested in another related question: 
\begin{question} \label{realizequestion2}
Let $\ell >h$ and $J\subseteq \Delta$ with Levi decomposition 
of ${\mathfrak g}={\mathfrak u}_{J}\oplus {\mathfrak l}_{J}\oplus {\mathfrak u}_{J}^{+}$. 
Does there exist a $U_{\zeta}({\mathfrak g})$-module $M$ such that 
$\HH^{\bullet}(u_{\zeta}({\mathfrak g}),M)$ identifies as a $G$-module 
with $\text{ind}_{P_{J}}^{G} S^{\bullet}({\mathfrak u}_{J}^{*})\otimes \bfk_{\gamma}$ where $\gamma\in X_{P_{J}}\cap X^{+}$?
\end{question} 
The paper is devoted to answering Questions \ref{realizequestion1} and \ref{realizequestion2} for quantum groups. We show that Question \ref{realizequestion2} has a positive answer in full 
generality. When $\Phi$ is of type $A_{n}$, Question \ref{realizequestion1} has an affirmative answer for all $J\subseteq \Delta$. For other root systems  
one needs to impose conditions involving the moment map and the normality of the orbit closure to guarantee the realization. 
In the process of answering Questions \ref{realizequestion1} and
\ref{realizequestion2} we construct
modules over the small quantum group
whose cohomology identifies with 
the global sections of the twisted cotangent bundle on $G/P_{J}$. The
resulting theorem can be viewed as a
generalization of the 
aforementioned result of Ginzburg and Kumar \cite{GK} and of  Kumar,
Lauritzen, and Thomsen \cite{KLT} to the
quantum and parabolic settings. The
computation of cohomologies in \cite{GK}
is a special case of what we computed here. However, \cite{GK} deals with the
Borel subalgebra, for which the Levi
subalgebra $\mathfrak l_J$ is the Cartan
subalgebra $ \mathfrak h$. When the Levi
subalgebra is not torus, the towers of algebras by  regular sequences in the
center of 
$\mathcal U_\zeta(\mathfrak u_J)$ are
not $U_\zeta(\mathfrak l_J)$-modules in
contrast to case of 
$ \mathfrak l_J=\mathfrak h$. 
We had to search for a completely different approach to deal with
$U_\zeta(\mathfrak p_J)$-module structures on the cohomology by proving
a theorem that, for  two 
$U_\zeta(\mathfrak p_J)$-modules $V_1$
and $V_2$, a linear isomorphism
$\phi:V_1\to V_2$ is a
$U_\zeta(\mathfrak p_J)$-module
isomorphism if and only if it is a
$U_\zeta(\mathfrak b)$-module
isomorphism. This property is more
general than what we need in this paper 
and uses deep properties of the generalized
tensor identity and the completeness of
the flag varieties. 

We should also mention that our calculations involve computing the ring structure of $\text{Ext}^{\bullet}_{u_{\zeta}({\mathfrak g})}(D,D)$ for 
a specific $U_{\zeta}({\mathfrak g})$-modules $D$. There is very little known about these $\text{Ext}$-algebras in general and we hope that our 
techniques and explicit computations will lead to a better understanding of these structures. It should be noted that the presence of strong cohomological 
vanishing results in the complex reductive group case allows us to make computations for the $\text{Ext}$-algebra in the quantum setting. 

It should be observed that Questions~\ref{realizequestion1} and \ref{realizequestion2} are still not completely resolved for reductive groups in positive characteristic. For 
Richardson orbits, Question~\ref{realizequestion1} holds by \cite{CLNP}, and for arbitrary orbits in classical groups in \cite{NT} 
when the field has good characteristics. Partial results for exceptional groups are provided in \cite{NT}. We will indicate the necessary vanishing results needed to establish positive answers to Questions \ref{realizequestion1} and \ref{realizequestion2} over fields of positive characteristics. 

\subsection{} The outline of the paper is as follows. In Section 2, we review the  construction 
of quantum groups for ${\mathfrak g}$ and its parabolic subalgebras. For our purposes these are 
important considerations to construct  suitable $U_{\zeta}({\mathfrak g})$-modules. In Section 3, we prove a general theorem 
which gives sufficient conditions to realize the global sections of the twisted cotangent bundle as the cohomology of the small quantum group with coefficients in a 
$U_{\zeta}({\mathfrak g})$-module $N$. In Section 4 we produce a module $N$ which gives a non-shifted realization of these global sections and as an application we
provide an affirmative answer to Question~\ref{realizequestion1}.  Later in Section 5, we apply the theory of tilting modules to yield a shifted realization of the global 
sections. This shifted realization result is a generalization of both the quantum and the parabolic versions of the work of Kumar, Lauritzen and Thomsen \cite{KLT}. 

Section 6 is devoted to investigating the connections between $G$-composition factors in the coordinate rings with 
quantum group cohomology. We demonstrate that computing certain $\text{Ext}$-groups for modules over the 
large quantum groups is equivalent to determining the $G$-composition factors in the modules $\text{ind}_{P_{J}}^{G} S^{\bullet}({\mathfrak u}_{J}^{*})\otimes \bfk_{\gamma}$ for a character $ \gamma $ of $P$. 
It is well-known that the later problem involves Kostant's Partition Function which can be related to  the coefficients of certain Kazhdan-Lusztig Polynomials (cf. \cite{Br2}).  In Section 7, we indicate how one can extend our results to Frobenius kernels under suitable cohomological vanishing assumptions in positive characteristic cases. Finally, in Section 8, we include an appendix which gives details on Hopf algebra actions on cohomology rings. The results are of independent interest and 
form the foundational basis for the work in the paper. For instance, in Section~\ref{parabolic}, we show that two modules for the algebraic group (quantum group, or more general Hopf algebras) are isomorphic if and only if they are isomorphic when restricted to a particular subalgebra provided the (rational) induced functor sends the trivial module for the subalgebra to the trivial module for the larger algebra. Thus many equivariant questions can be reduced to equivariant questions for a Borel subgroups (subalgebras). \subsection{Acknowledgments} The authors dedicate this paper to Georgia Benkart and Brian Parshall for their many contributions in the area of Lie and representation theory. The authors (along with Jens C. Jantzen) were honored to serve with Georgia and Brian on the organizing committee for the 2004 AMS Summer Research Conference in Snowbird, Utah, celebrating James E. Humphreys' 65th birthday. Unfortunately, he also passed away during the pandemic in 2020. The authors thank the referee for a careful reading of the paper.

\section{Quantum Groups and Hopf Algebra Actions}

\subsection{} \label{sec:2.1} We will follow the conventions as described in 
\cite[Section 2]{BNPP}. Let $G:=G_{\mathbb Q}$ be a simply connected simple
algebraic group which is defined and split over ${\mathbb Q}$ with Lie algebra ${\mathfrak g}$ over $\mathbb Q $ with a fixed Chevalley basis. 
Let $\Phi$ be the irreducible root system associated to ${\mathfrak g}$ with respect to a fixed maximal split torus $T$ of $G$.
Let $\Delta$ be the corresponding set of simple roots. The set $\Phi$ spans a real
vector space $\mathbb E$ with positive definite inner product
$\langle u,v\rangle$, $u,v\in \mathbb E$, adjusted so that
$\langle\alpha,\alpha\rangle=2$ if $\alpha\in\Phi$ is a short root.
If $\alpha\in\Phi$ then let $\alpha^\vee=\frac{2}{\langle\alpha,\alpha\rangle}\alpha$ be the coroot. 
For $J\subseteq \Delta$, let $\Phi_{J}=\Phi\cap {\mathbb Z}J$ be the root system of 
$\Phi$ generated by $J$. Let $W$ be the Weyl group corresponding to $T$. For $J\subset \Delta$,  let 
$W_{J}$ be the Weyl group of $\Phi_{J}$, viewed as a subgroup of $W$, and ${^J}W$ be
the set of minimal length coset representatives for $W_{J}\backslash W$. Let $ w_0\in W$ be the unique longest element. We will use $ w_{0, J}\in W_J$ to denote the unique longest element of $ W_J$. 

Define the fundamental dominant weights $\omega_{\alpha_1},\cdots,\omega_{\alpha_n}$
by $\langle\omega_{\alpha_{i}},\alpha_j^\vee\rangle=\delta_{i,j}$ for all $\alpha_j\in \Delta$, so that the weight lattice $X=X(T)={\mathbb Z}\omega_{\alpha_1}\oplus \cdots\oplus{\mathbb Z}\omega_{\alpha_n}$ and the set of dominant weights $X^+={\mathbb N}\omega_{\alpha_{1}}\oplus\cdots\oplus {\mathbb N}\omega_{\alpha_n}$.

Let ${\mathfrak t}$ be a Lie algebra of $T$ which is a Cartan subalgebra of $\mathfrak g$. 
Given $\alpha\in\Phi$, let ${\mathfrak g}_\alpha$ be the $\alpha$-root space. Put ${\mathfrak
b}^+=\mathfrak t\oplus\bigoplus_{\alpha\in\Phi^+}{\mathfrak
g}_\alpha$ (the positive Borel subalgebra), and ${\mathfrak b}=
\mathfrak t\oplus\bigoplus_{\alpha\in\Phi^-}{\mathfrak g}_\alpha$
(the Borel subalgebra opposite to $\mathfrak{b}^+$). We will denote by $B^+$ and $ B$ the corresponding Borel subgroups of $G$ containing the maximal torus $T$. More generally, given a subset $J \subseteq \Delta$, 
one can consider the  parabolic subgroup $P_J$ of $G$ containing $B$ with a fixed Levi subgroup $L_J$ (containing $T$)  such that their Lie algebras  are $\Lie(L_J)={\mathfrak l}_J$ and 
$\Lie(P_J)={\mathfrak p}_J = {\mathfrak l}_J\oplus{\mathfrak u}_J$ of $\mathfrak g$. Note that $\Phi_J$ is the root system of $L_J$.  
In this case, set $X_{P_{J}}=\{\lambda\in X:\ \langle \lambda, \alpha^{\vee} \rangle =0\  \text{for all $\alpha\in J$}\}$. Observe that 
$X_{P_{J}}$ identifies with the one-dimensional representations of $L_{J}$ or the  the character group  $\hom(L_J, G_m)$.  

\subsection{} Throughout this paper let $\ell >1$ be a fixed odd positive 
integer. If $\Phi$ has type $G_2$, then we assume that $3$ does not divide 
$\ell$. Let ${\mathcal A}={\mathbb Z}[q,q^{-1}]$ with fraction field ${\mathbb Q}(q)$. 
Let $\zeta\in\mathbb C$ be a primitive $\ell$th root of unity
and $\mathsf k={\mathbb Q}(\zeta)$. One can regard $\bfk$ as an $\mathcal
A$-algebra by means of the homomorphism $\mathbb{Z}[q,q^{-1}]\to \bfk$ where 
$q\mapsto \zeta$.

The quantized enveloping algebra $\BU_q({\mathfrak g})$ of
$\mathfrak g$ is the ${\mathbb Q}(q)$-algebra with generators 
$E_{\alpha}$, $K_{\alpha}^{\pm 1}$, $F_{\alpha}$, $\alpha\in \Delta$ and
relations (R1)---(R6) listed in \cite[(4.3)]{Jan2}.
The algebra $\BU_q({\mathfrak g})$ has two $\mathcal A$-forms. One is
$U_q^{\mathcal A}({\mathfrak g})$ defined by Lusztig \cite{lusztig:root} as  the $ \mathcal A$-subalgebra of $\BU_q({\mathfrak g})$ generated by 
$E^{(n)}_{\alpha}=E^{n}_{\alpha}/[n]!$, $K_{\alpha}^{\pm 1}$, $F^{(n)}_{\alpha}=F^{n}_{\alpha}/[n]!$. The other is 
${\mathcal U}_q^{\mathcal A}({\mathfrak g})$ defined  De
Concini and Kac \cite{DK}) as the $\mathcal A$-subalgebra of $\BU_q({\mathfrak g})$ generated by $ \{ E_\alpha, K_\alpha^{\pm 1}, F_{\alpha}\;|\; \alpha \in \Delta\}$. Both are free as $\mathcal A$-modules and  Hopf $\mathcal A$-subalgebras of  $\BU_q({\mathfrak g})$. Since  ${\mathcal U}_q^{\mathcal A}({\mathfrak g})\subseteq U_q^{\mathcal A}({\mathfrak g})$, applying the functor $ \bfk \otimes_{\mathcal A}-$, one obtains a Hopf $ \bfk$-algebra homomorphism $\mathcal U_\bfk(\mathfrak g)=\bfk\otimes_{\mathcal A}{\mathcal U}_q^{\mathcal A}({\mathfrak g})\to \bfk\otimes_{\mathcal A}U_q^{\mathcal A}({\mathfrak g})=U_\bfk({\mathfrak g})$. 
 These two Hopf $\bfk$-algebras $\mathcal U_\bfk(\mathfrak g) $ and $U_\bfk({\mathfrak g})$ play the roles
analogous to the the universal
enveloping algebra of its Lie algebra, respectively and  hyperalgebra (algebra of distributions) of a reductive group over fields of positive characteristics \cite{Jan1}.
 
\delete{Set 
\begin{equation*}\label{firstquantumgroup}
U_\zeta({\mathfrak g}):=k\otimes_{\mathcal A}U^{\mathcal A}_q({\mathfrak g})/\langle1\otimes 
K_\alpha^l-1\otimes 1,\alpha\in\Delta \rangle,\end{equation*}
\zl{lusztig does not take the quotient for quantum group, same for $\mathcal U_\zeta$}
 where $\langle \cdots \rangle$ denotes ``the ideal generated'' by the elements within the brackets. }
  Let $\BU_{q}^{0}$ be the subalgebra of $\BU_q({\mathfrak g})$  generated by 
the $\{K_\alpha^{\pm 1}:\ \alpha\in \Delta\}$.  Then $\BU_{q}^{0}$ is exactly the Laurent polynomial ring $\mathbb Q(q)[K_\alpha ^{\pm 1}\; |\; \alpha \in \Delta]$. Let $ \mathcal U^{0, \mathcal A}=\mathcal U^{ \mathcal A}(\mathfrak g)\cap \BU_{q}^{0}$ and $  U^{0, \mathcal A}= U^{ \mathcal A}(\mathfrak g)\cap \BU_{q}^{0}$. We have $ \mathcal U^{0, \mathcal A}\subseteq  U^{0, \mathcal A}$. Similarly, one has a Hopf algebra homomorphism $  \mathcal U^{0}_{\bfk}\to  U^{0}_\bfk$. 

The elements $E_\alpha,\delete{K_\alpha,} F_\alpha$, $\alpha\in\Delta$, in $U_{\bfk}({\mathfrak g})$ generate a
finite dimensional Hopf subalgebra, denoted by $u_\zeta({\mathfrak g})$, of $U_\bfk({\mathfrak g})$.
The Hopf algebra $u_{\zeta}({\mathfrak g})$ will be referred to as the {\em small quantum group}. We remark that $u_{\zeta}({\mathfrak g})$ is 
a normal  Hopf subalgebra of $U_\bfk({\mathfrak g})$ such that 
$U_\bfk({\mathfrak g})/\!\!/u_{\zeta}({\mathfrak g})=U_\bfk({\mathfrak g})/\langle \ker(\epsilon: u_\zeta(\mathfrak g)\to k)\rangle $ is a Hopf algebra \cite{Lin2} and the $(U_{\bfk}({\mathfrak g})/\!\!/u_{\zeta}({\mathfrak g}))/\langle K_\alpha ^l-1\;|\; \alpha\in \Delta\rangle \cong {\mathcal U}({\mathfrak g_\bfk})$, 
where ${\mathcal U}({\mathfrak g}_k)$ is the ordinary universal enveloping algebra ${\mathfrak g}_k=\mathfrak g\otimes_{\mathbb Q}k$ over $ k$, see \cite{lusztig:root,lusztig:finite-dim} for more details.

We will also modify some of the algebras a little and define 
 \begin{equation*}\label{secondquantumgroup}
{\mathcal U}_{\zeta}({\mathfrak g})= 
{\mathcal U}_{\bfk}({\mathfrak g})/\langle  K_\alpha^l-1,\alpha\in\Delta\rangle
\end{equation*}
and denote $U_\zeta(\mathfrak g)=U_\bfk(\mathfrak g)$.
In this paper we will use the same letters $ E_\alpha$, $F_\alpha$, etc, for elements in any one of the algebra $\mathcal U_\bfk(\mathfrak g)$, $ \mathcal U_\zeta(\mathfrak g)$, $   U_\zeta(\mathfrak g)$, or in $   u_\zeta(\mathfrak g)$ once the context is understood.

The (Hopf) algebra ${\mathcal U}_\zeta({\mathfrak g})$ has a central subalgebra
${\mathcal Z}$ which is  generated by $\{ E_\alpha^l, F_\alpha^l\;|\; \alpha\in \Delta\}$ and is a normal Hopf $\bfk$-subalgebra  such that $u_{\zeta}({\mathfrak g})\cong{\mathcal
U}_\zeta({\mathfrak g})/\!\!/{\mathcal Z}$ (cf. \cite[Cor. 3.1]{DK} for more details). 

We will assume throughout the paper that the $U_\zeta({\mathfrak g})$-modules $M$ are 
{\em integrable of type 1} (cf. \cite[Section 2.2]{BNPP}). Let $\text{Fr}:U_{\zeta}({\mathfrak g})\rightarrow 
{\mathcal U}({\mathfrak g}_\bfk)$ be the (quantum) Frobenius homomorphism. 
If $N$ is a (locally finite) ${\mathcal U}({\mathfrak g}_\bfk)$-module 
then the $U_{\zeta}({\mathfrak g})$-module $N^{[1]}$ 
(integrable, type 1) is  the inflation of $N$ by $\on{Fr}$. 
Conversely, for any $U_\zeta(\mathfrak g)$-module $M$ 
which is $u_{\zeta}(\mathfrak g)$-trivial, 
then $M\cong N^{[1]}$ for some  
$ \mathcal U(\mathfrak g)$-module $N$.  

\subsection{\bf Levi and Parabolic Subalgebra}\label{PBW}
For each $\alpha \in \Delta$, Lusztig has defined an automorphism
$T_{\alpha}$ of $\BU_q({\mathfrak g})$ (cf. \cite[Ch. 8]{Jan3}, \cite{lusztig:finite-dim}). If $s_{\alpha}$ 
is a simple reflection in $W$, let $T_{s_{\alpha}} := T_{\alpha}$. 
More generally, given any $w \in W$, let $w = s_{\beta_1}s_{\beta_2}\cdots s_{\beta_n}$ be
a reduced expression, and define $T_w := T_{\beta_1}\cdots T_{\beta_n}
\in \text{Aut}(\BU_q({\mathfrak g}))$. The automorphism $T_w$ is independent of the
reduced expression of $w$ \cite[Thm 3.2]{lusztig:root}.   
Now for each ${\gamma}\in \Phi^{+}$, 
one can define
$E_{\gamma}=T_w(E_{\beta})$, 
$F_{\gamma}=T_w(F_{\beta})$  for 
$ w\in W$ with $w(\beta)\in \Phi^+$ 
and $ \beta\in \Delta$ \cite[Prop. 1.8]{lusztig:finite-dim} 
and \cite[Section 2.4]{BNPP}.  
Note that $E_{\gamma}$ has weight $\gamma$, and 
$F_{\gamma}$ has weight $-\gamma$. 

Now let $J \subseteq \Delta$ and fix a 
reduced expression $w_0 =s_{\beta_1}\cdots s_{\beta_N}$ that starts with a 
reduced expression for the long element $w_{0,J}$ for $W_{J}$. If $w_{0,J} =
s_{\beta_1}\cdots s_{\beta_M}$, then $s_{\beta_{M+1}}\cdots s_{\beta_N}$ is a
reduced expression for $w_J=w_{0,J}w_0$. Let ${\mathfrak l}_{J}$ be the Levi subalgebra 
corresponding to $J$ and ${\mathfrak p}_{J}$ the parabolic 
subalgebra containing $ \mathfrak{b}$ and $\mathfrak{l}_J$. The universal enveloping algebras (over the field $\bfk$) will be denoted by ${\mathcal U}({\mathfrak l}_{J})$ and
${\mathcal U}({\mathfrak p}_{J})$. The Lie algebras considered here will be over the field $ \bfk$ without adding the subscript $ \bfk$ (despite the fact it was assumed that $ \mathfrak g$ is a Lie algebra over $ \mathbb Q$). 

One can naturally define corresponding quantized
enveloping algebras $\BU_q({\mathfrak l}_{J})$ and $\BU_q({\mathfrak p}_{J})$ over $\mathbb Q(q)$ as subalgebras of $\BU_q(\mathfrak g)$.  The subalgebra $\BU_q({\mathfrak l}_{J})$ is generated by $\{E_{\alpha},
F_{\alpha} : \alpha \in J\} \cup \{K_{\alpha}^{\pm 1} : \alpha \in \Pi\}$, and
$\BU_q({\mathfrak p}_{J})$ is the subalgebra generated by $\{E_{\alpha}: \alpha \in J\}
	\cup \{F_{\alpha}, K_{\alpha}^{\pm 1} : \alpha \in \Pi\}$. Using the PBW type of basis constructed by Lusztig in \cite{lusztig:finite-dim}, these subalgebras has two different $ \mathcal{A}$-forms as pure subalgebras of $U_{q}^{\mathcal{A}}(\mathfrak{g})$ and $\mathcal{U}_{q}^{\mathcal{A}}(\mathfrak{g})$ respectively.   Upon specialization of $q$ to $ \zeta$ as in Section~\ref{sec:2.1}, one obtains 
the subalgebras $U_{\zeta}({\mathfrak l}_{J})$, $U_{\zeta}({\mathfrak p}_{J})$, 
$u_{\zeta}({\mathfrak l}_{J})$, $u_{\zeta}({\mathfrak p}_{J})$ of
$U_\zeta({\mathfrak g})$, and $\mathcal{U}_{\zeta}({\mathfrak l}_{J})$ and
$\mathcal{U}_{\zeta}({\mathfrak p}_{J})$ of ${\mathcal U}_\zeta({\mathfrak g})$.
One can also make analogous constructions with the opposite parabolic ${\mathfrak p}^{+}_{J}$.  The subalgebras  $U_{\zeta}({\mathfrak l}_{J})$ and $U_{\zeta}({\mathfrak p}_{J})$ were first constructed in \cite{APW}. 

We remark that the algebras $U_{\zeta}({\mathfrak l}_{J})$,  $U_{\zeta}({\mathfrak p}_{J})$, $\mathcal{U}_{\zeta}({\mathfrak l}_{J})$, and
$\mathcal{U}_{\zeta}({\mathfrak p}_{J})$ are Hopf algebras. The inclusion  $\mathcal{U}_q^{\mathcal{A}}(\mathfrak{g})\subseteq U_q^{\mathcal{A}}(\mathfrak{g}) $ of $\mathcal{A}$-algebras induces a homomorphism of Hopf algebras: $ \mathcal{U}_\zeta(\mathfrak{g})\rightarrow U_\zeta(\mathfrak{g}) $ with image being $u_{\zeta}(\mathfrak{g})$. This Hopf algebra homomorphism induces Hopf algebra homomorphisms $ \mathcal{U}_\zeta( \star)\rightarrow U_\zeta(\star) $  with $ \star ={\mathfrak l}_{J}, {\mathfrak p}_{J}$

The above  reduced expression
for $w_0$ (beginning with one for $w_{0,J}$) defines an ordering on  positive roots
$\Phi^+=\{\gamma_1, \cdots \gamma_N\}$, with $\{\gamma_1, \cdots, \gamma_{N_J}\}=\Phi^+_J$ being the positive roots of the reductive Lie algebra $\mathfrak l_J$ and $\{\gamma_{N_J+1},\cdots,  \gamma_N\}$ being the roots in the nilpotent radical $ \mathfrak u^+_J$ of $ \mathfrak p^+_J$. 
With the PBW basis $\{F_{\gamma_{1}}^{(a_{1})}\cdots F_{\gamma_N}^{(a_N)}\;|\; ( a_1, \cdots, a_N)\in \mathbb N^N\}$ of $\BU_q(\mathfrak n)$ as described in \cite[Appendix]{lusztig:root} using this reduced expression of $w_0$,  one can define
a subalgebra $\BU_q({\mathfrak u}_{J})$ of $\BU_q(\mathfrak n)\subseteq \BU_q(\mathfrak p_J)$ generated by all $F_{\gamma_i}$ with $ i>N_J$.  $\BU_q({\mathfrak u}_{J})$ is an augmented normal subalgebra of $\BU_q(\mathfrak{p}_J)$  (although not a Hopf subalgebra of $\BU_{q}(\mathfrak{g})$). The algebra $\BU_q({\mathfrak u}_{J})$ is  analogous to 
${\mathcal U}({\mathfrak u}_{J}) \subseteq {\mathcal U}({\mathfrak p}_{J})$. 

  The $ \mathbb Q(q)$-subalgebra $\BU_q({\mathfrak u}_{J})$ in $ \BU_q(\mathfrak g)$ is spanned by the $F_{\gamma_{N_J+1}}^{(a_{N_J+1})}\cdots F_{\gamma_N}^{(a_N)}$, $a_i \in
\mathbb{N}$. According to \cite[Lemma 2.4.1]{BNPP} $\BU_q({\mathfrak u}_{J})$ is a
subalgebra of $\BU_q({\mathfrak p}_{J})$ and independent of the choice of reduced expression for
$w_0$. Again by specializing $q$ to $ \zeta$, one obtains algebras $U_{\zeta}({\mathfrak u}_{J})$ and 
$u_{\zeta}({\mathfrak u}_{J})$ as subalgebras of $U_\zeta(\mathfrak{g})$ and $ \mathcal U_\zeta(\mathfrak u_J)$ as a subalgebra of $\mathcal U_\zeta(\mathfrak p_J)$. 
Again, the Hopf algebra homomorphism $ \mathcal{U}_\zeta(\mathfrak{g})\rightarrow U_\zeta(\mathfrak{g}) $ restricts to an algebra homomorphism $ \mathcal{U}_\zeta( \mathfrak{u}_J)\rightarrow U_\zeta(\mathfrak{u}_J) $ with image being $ u_\zeta(\mathfrak u_J)$.

\subsection{Adjoint action}\label{sec:2.4} Since $\mathcal U_\zeta(\mathfrak g)$ is a Hopf algebra, and $E_{\beta}^l, F_\beta^l$, and $ K_i^l$ are central elements, then the adjoint action of $\mathcal U_\zeta(\mathfrak g)$ on the central elements is trivial in the sense that $ u\cdot z=\epsilon(u)z$ for all $ u\in \mathcal U_\zeta(\mathfrak g)$ and $ z$ central in $\mathcal U_\zeta(\mathfrak g)$.  Thus, the adjoint action of  elements $ E_\beta^l$ and $ F_\beta^l$ on $\mathcal U_\zeta(\mathfrak g)$ is zero and $ K_i^l$ acts on $\mathcal U_\zeta(\mathfrak g)$ as the identity. Therefore, the adjoint action of $\mathcal U_\zeta(\mathfrak g)$ on $\mathcal U_\zeta(\mathfrak g)$ factors through $u_\zeta(\mathfrak g)$. In particular, an  element $u\in u_\zeta(\mathfrak p_J)$ acts on $ f^l$ via the counit $ \epsilon(u)$ for $ f=E_\beta, F_\beta, K_i$.  Furthermore, for any subalgebra $A$ of $\mathcal U_\zeta(\mathfrak g)$ we can also construct the adjoint action of $A$ on certain subspaces  of $\mathcal U_\zeta(\mathfrak g)$. This action by $A$ always factors through its image in $ u_\zeta(\mathfrak g)$. 

To extend the adjoint action of $ u_\zeta(\mathfrak g) $ on $\mathcal U_\zeta(\mathfrak g)$ we have to return to the $\mathcal A$-forms. 
Both $ \mathcal{A}$-forms $ U_{q}^{\mathcal{A}}(\mathfrak{g})$ and $ \mathcal{U}_q^{\mathcal{A}}(\mathfrak{g})$ are  Hopf  $\mathcal{A}$-subalgebras of $\BU_q(\mathfrak{g})$. The following proposition is from {\cite[Prop. 2.9.2]{ABG}}. 
For the definition and properties of the adjoint action of Hopf algebras we refer the reader to Section~\ref{Hopf-algebra-action}. 

\begin{proposition}\label{2.4.1} Let $ \mathcal A_\zeta$ be the localization of $\mathcal A$ at the maximal ideal generated by $ q-\zeta$. The algebra $\mathcal{U}_q^{\mathcal{A}_\zeta}(\mathfrak{g})$ is invariant under the left adjoint action of $ {U}_q^{\mathcal{A}_\zeta}(\mathfrak{g})$. 
In particular,  the algebra $\mathcal{U}_\zeta(\mathfrak{g})$ is a left   $ {U}_\zeta(\mathfrak{g})$-module algebra.
\end{proposition}
  
The proof in \cite[Prop. 5.3]{Lin1} shows that  $u_{\zeta}({\mathfrak u}_{J})$ is invariant under the left adjoint action of $U_\zeta(\mathfrak{p}_J)$.  The above proposition 
defines a  $U_\zeta(\mathfrak{p}_J)$-module algebra structure on $\mathcal{U}_{\zeta}(\mathfrak {u}_{J})$ such that the homomorphism $ \mathcal{U}_{\zeta}({\mathfrak u}_{J})\rightarrow u_{\zeta}({\mathfrak u}_{J})$ 
is a homomorphism of $U_\zeta(\mathfrak{p}_J)$-module algebras. Since the algebra  $\mathcal{U}_{\zeta}({\mathfrak u}_{J})$ is generated by $ F_{\gamma_i}$ for $i=N_J+1, \dots, N$, one only needs to define the 
action on these elements. The algebra $U_\zeta(\mathfrak{p}_J)$ is generated as an algebra by $ E_{\alpha}, E_{\alpha}^{(l)}$ and $ F_{\beta},\; F_{\beta}^{(l)}$ with $ \alpha \in J $ and $\beta \in \Delta$ together with $ U_\zeta^0$.  
The action of $E_{\alpha}, F_{\beta}, K_{\beta}$  on $ \mathcal{U}_\zeta(\mathfrak{u}_J)$ can be expressed by writing down the comultiplication. We only consider the action of $ F_{\beta}^{(l)}$ and $E_{\alpha}^{(l)}$. In the formal definition of the action 
of  $ F_{\beta}^{(l)}$, one follows the argument in \cite[Prop. 5.3]{Lin1} together with the argument of \cite[Lemma 8.5]{lusztig:root}. The action of $E_{\alpha}^{(l)}$ follows a similar argument in \cite[5.1]{Lin2}.

\subsection{} The Hopf  subalgebras $U^0_\zeta$  and  $U^0_\zeta(\mathfrak p_J)$ act on the algebras $ \mathcal U_\zeta(\mathfrak u_J)$ and $ U_\zeta(\mathfrak u_J)$. Moreover, the algebra homomorphism $ \mathcal U_\zeta(\mathfrak u_J)\to  U_\zeta(\mathfrak u_J)$ is a homomorphism of  $U_\zeta(\mathfrak p_J)$-module algebras. In this paper, we will need to discuss $U^0_\zeta$, $ U_\zeta(\mathfrak l_J)$, and $U_\zeta(\mathfrak p_J)$-equivariant $\mathcal U_\zeta(\mathfrak u_J)$-modules and $ U_\zeta(\mathfrak u_J)$-modules.  This is discussed in the appendix. By $ U^0_\zeta$ and $ U_\zeta(\mathfrak l_J)$-modules, we mean the integrable of type 1 modules, which are locally finite and with a weight space decomposition as described in \cite{APW} and \cite{Lin1}. 

\section{Cohomological Calculations}

\subsection{}\label{sec:3.1} In this section we will present a general method for making cohomological computations for quantum groups given two key assumptions. 
The first assumption appears in the next proposition and involves computing cohomology for ${\mathcal U}_{\zeta}({\mathfrak u}_{J})$. Note that although ${\mathcal U}_{\zeta}({\mathfrak u}_{J})$ is not a Hopf algebra, it is an augmented algebra and the cohomology we consider here is the cohomology of the augmented (supplemented) algebra in the sense of Cartan-Eilengberg \cite[Ch X]{CE}.  Once this 
assumption is satisfied one can compute the corresponding cohomology for $u_{\zeta}({\mathfrak u}_{J})$. The proof uses the constructions presented in Section 2 and employs the 
techniques outlined in the proofs of \cite[Theorem 5.3.1, Lemma 5.4.1]{BNPP}. The proof will proceed by induction on successive quotients of
$U_{\zeta}({\mathfrak u}_{J})$ (cf. \cite[2.4]{GK}) and can be described as follows. 

Let $N^J = |\Phi^+ \backslash\Phi_J^+|$ (thus $ N_J+N^J=N$) and choose a fixed ordering of root vectors $f_1$, $f_2$, \dots, $f_{N^J}$
in ${\mathcal U}_{\zeta}({\mathfrak u}_{J})$ corresponding to the positive roots $\{\gamma_{1},\gamma_{2},\dots, \gamma_{N^J}\}$ 
in $\Phi^+\backslash\Phi_J^+$ such 
that for each $i$, the subalgebra generated by $ \langle f_1, \cdots, f_i\rangle $ is 
$ U_\zeta(\mathfrak b)$-stable through
the action of $ U_\zeta(\mathfrak g)$ 
on $\mathcal U_\zeta(\mathfrak g)$. 
Each $f_i^l$ is central in 
$ \mathcal{U}_{\zeta}(\mathfrak{g})$ and is contained in the 
augmented ideal of ${\mathcal U}_{\zeta}({\mathfrak u}_{J})$. 
Let $Z_{i}$  be the subalgebra of $ \mathcal{U}_\zeta(\mathfrak{g})$
 generated by 
 $\langle f_{1}^{l},f_{2}^{l},\dots,f_{i}^{l} \rangle$ 
 (with $i=1, \cdots, N^J$). 
 Then $ Z_i$ is a central subalgebra of 
 $ \mathcal{U}_\zeta(\mathfrak{u}_J)$ 
 and ${u}_\zeta(\mathfrak{u}_J)
 =\mathcal{U}_\zeta(\mathfrak{u}_J) /\!\!/ Z_{N^J}$.  
 We note that for each $i$, $ Z_i$ are 
 $ U_\zeta(\mathfrak b)$-stable subalgebras of
 $ \mathcal U_\zeta(\mathfrak u_J)$, but in general, $Z_i$ are not 
 $ U_\zeta(\mathfrak p_J)$-stable. Note that $Z_{N^J}$ is $U_\zeta(\mathfrak p_J)$-stable.

 For $0 \leq i \leq N^J$, 
 let $A_i = {\mathcal U}_{\zeta}({\mathfrak u}_{J})
 /\langle f_1^l,f_2^l,\dots,f_i^l\rangle$with 
$A_{N^J} = u_{\zeta}({\mathfrak u}_{J})$.  
Note that  each $ A_i$ is a
$U_\zeta(\mathfrak b)$-module algebra.
Furthermore, for $1 \leq i \leq N^J$, 
let $B_i = \langle f_i^l\rangle \subseteq A_{i-1}$ 
be the augmented subalgebra
generated by $f_i^l$.  
Each $B_i$ is a polynomial algebra in
one variable, central in $A_{i-1}$ 
and stable under $U_\zeta(\mathfrak b)$. 
Thus $B_i$ is a normal augmented subalgebra of $ A_{i-1}$ 
in the sense of \cite[XVI. \S 6]{CE}, and  $A_{i-1}/\!\!/B_i  \cong A_i$ as $U_\zeta^0$-modules.  By the discussion in Section~\ref{sec:2.4}, the ideals $\langle f_1^l,f_2^l,\dots,f_i^l\rangle$ are $u_\zeta(\mathfrak p_J)$-stable which induces an action of $ U_\zeta(\mathfrak b) u_{\zeta}(\mathfrak p_J)$  on $ A_i$. 

For $0 \leq i \leq N^J$, let $V_i$ be a vector space with basis
$\{x_1,x_2,\dots,x_i\}$ considered as a $U_{\zeta}(\mathfrak b)$-module by letting $x_i$ have weight $-\gamma_i$ contained as a $U_{\zeta}(\mathfrak b)$-submodule in  $V_N
\cong {\mathfrak u}_{J}^*$ with $V_0 = \{0\}$.   For each integrable $ U_\zeta(\mathfrak{p}_J)$-module $M$ of type 1 (in the sense of \cite{APW}), the space $\hom_{u_{\zeta}(\mathfrak{u}_J)}(\bfk , M)$ is  a rational $P_J$-module for the algebraic group $ P_J$ (the parabolic subgroup of $G$ defined and split over $\bfk$) with differential being the $\mathcal{U}(\mathfrak{p}_J)$-module arising from the Hopf algebra isomorphism  
 $\mathcal{U}(\mathfrak{p}_J)\cong U_{\zeta}(\mathfrak{p}_J)/\!\!/u_{\zeta}(\mathfrak{p}_J)$.  Any $ U_{\zeta}(\mathfrak{p}_J)$-module $M$ restricts to a $ u_{\zeta}(\mathfrak{p}_J)$-module. The module can then be inflated to a $\mathcal{U}_\zeta(\mathfrak{p}_J)$-module via the Hopf algebra homomorphism $\mathcal U_\zeta(\mathfrak p_J)\to U_\zeta(\mathfrak p_J)$. Then the module can be regarded as a $ \mathcal{U}_\zeta(\mathfrak{u}_J)$-module where  $ Z_J^+ M=0$. Let $Z_J=Z_{N^J}$ be the polynomial algebra with generators $ f_i^l$ and $ Z_J^+=\langle f_1^l, \cdots, f_{N^J}^l\rangle$ be the maximal ideal of $ Z_J$.

\subsection{} We recall the following general setting. For a given Hopf algebra $H$ and a $H$-module algebra $A$, we can consider the category of  $HA$-modules as the $A$-modules in the tensor category of $H$-modules.  The objects are $H$-modules $M$ with an $A$-module structure $ A\otimes_\bfk M\to M$, which is an $H$-module homomorphism. The morphisms are the $\bfk$-linear maps that are homomorphisms for both $A$-modules and $ H$-modules. If $ M, N$ are two $ HA$-modules, then $\hom_A(M,N)$ is an $H$-module. If $ A$ is an augmented algebra with the augmentation $A\to \bfk$ being an $ H$-module homomorphism, then $H$ acts on the cohomology groups $\operatorname{H}^k(A, M)=\on{Ext}^k_A(\bfk, M)$ for any $ HA$-module $M$. 

 For $ H=U_\zeta(\mathfrak p_J)$, $ u_{\zeta}(\mathfrak u_J)$ is an augmented $U_\zeta(\mathfrak p_J)$-module algebra. According to Lemma~\ref{2.4.1},  $\mathcal U_\zeta(\mathfrak u_J)$ is also a  $U_\zeta(\mathfrak p_J)$-module algebra and  the augmented algebra homomorphism $ \mathcal U_\zeta(\mathfrak u_J)\to u_\zeta(\mathfrak u_J)$ is also a homomorphism of $U_\zeta(\mathfrak p_J)$-module algebras. 
 We have two types of cohomology groups $ \on H^*( \mathcal U_\zeta(\mathfrak u_J), M)$ and $ \on H^*(  u_\zeta(\mathfrak u_J), M)$ for each $U_\zeta(\mathfrak p_J)$-module $M$.  Thus both $ \on H^*( \mathcal U_\zeta(\mathfrak u_J), M)$ and $ \on H^*(  u_\zeta(\mathfrak u_J), M)$ are $U_\zeta(\mathfrak p_J)$-modules, on which $u_\zeta(\mathfrak u_J)$ acts trivially. 
   But the subalgebras $A_i$ constructed above are only $U_\zeta(\mathfrak b)  $-module algebras.  By noting that $ U_\zeta (\mathfrak b)//u_{\zeta}( \mathfrak b) \cong \mathcal U(\mathfrak b)$, we know that $\on{Hom}_{u_{\zeta}({\mathfrak l}_{J})}({\bfk},\on{H}^{n}(A_i,M) \otimes Q))$ is a rational $ B$-module for any (integrable of type 1)
 $U_\zeta(\mathfrak p_J)$-modules $M$ and $Q$ as stated in the following proposition.

\begin{proposition}\label{quantumExtcalculation} Let $J\subseteq \Delta$.  Let $M$ and $Q$ be $U_{\zeta}({\mathfrak p}_{J})$-modules such that  
$Q$ is an injective $u_{\zeta}({\mathfrak l}_{J})$-module and is trivial as $U_{\zeta}({\mathfrak u}_{J})$-module. 
If there is an integer $t\geq 0$ such that there are rational $P_J$-module isomorphisms 
\begin{equation*}
\on{Hom}_{u_{\zeta}({\mathfrak l}_{J})}({\bfk},\operatorname{H}^{n}({\mathcal U}_{\zeta}({\mathfrak u}_{J}), M)\otimes Q)\cong
\begin{cases} \bfk & n=t, \\
0 & \text{otherwise},
\end{cases} 
\end{equation*} 
then  we have the following rational $P_{J}$-module isomorphisms: 
\begin{itemize} 
\item[(a)] 
\begin{equation*} 
\on{Hom}_{u_{\zeta}({\mathfrak l}_{J})}({\bfk},\on{H}^{n}(u_{\zeta}({\mathfrak u}_{J}),M) \otimes Q)) 
\cong \begin{cases} S^{\frac{n-t}{2}}({\mathfrak u}^{*}_{J})  & n\equiv t\ (\operatorname{mod} 2),  \\
0 & \text{otherwise} ;
\end{cases} 
\end{equation*} 
\item[(b)] 
\begin{equation*}
\operatorname{H}^{n}(u_{\zeta}({\mathfrak p}_{J}),M\otimes Q)\cong
\begin{cases} S^{\frac{n-t}{2}}({\mathfrak u}^{*}_{J})  & n-t\equiv 0 \  (\operatorname{mod} 2), \\
0, & \text{otherwise} .
\end{cases} 
\end{equation*} 
\end{itemize} 
\end{proposition} 

\begin{proof} (a) Note that each  $U_{\zeta}(\mathfrak{p}_J)$-module $M$ 
becomes  a $ \mathcal{U}_\zeta(\mathfrak{u}_J)$-module satisfying $Z_JM=0$. 
 By Proposition~\ref{2.4.1}, $ \mathcal{U}_\zeta(\mathfrak{u}_J)$ is a 
 right $  U_\zeta(\mathfrak{p}_J)$-module algebra and the left 
 $ \mathcal{U}_\zeta(\mathfrak{u}_J)$-module structure on $M$ is compatible 
 with the left $  U_\zeta(\mathfrak{p}_J)$-module structure on $M$ by Section~\ref{sec:3.1}. Hence, the space $\hom_{\mathcal{U}_\zeta(\mathfrak{u}_J)}(\bfk, M) $ is 
 a $U_\zeta(\mathfrak{l}_J)$-module (and a $\mathcal{U}(\mathfrak{u}_J)$-module).  
 In particular, $\on{H}^i(\mathcal{U}_\zeta(\mathfrak{u}_J), M) $ is a 
 $ U_\zeta(\mathfrak{p}_J)$-module such that $ u_{\zeta}(\mathfrak{u}_J)$ 
 acts trivially by Proposition~\ref{triviality}.

Set  ${\mathcal G}(-) =\hom_{u_{\zeta}({\mathfrak l}_{J})}({\bfk},-\otimes Q)$.  
For each $ U_{\zeta}(\mathfrak{p}_J)$-module $E$ (integrable of type 1) on which 
$ u_{\zeta}(\mathfrak{u}_J)$-acts trivially, $\mathcal{G}(E)$ is a rational 
$ P_J$-module. Thus $ \mathcal{G}$ is a functor from the full subcategory of 
$U_{\zeta}(\mathfrak{p}_J)$-modules on which $u_{\zeta}(\mathfrak{u}_J)$ acts 
trivially to the category of rational $ P_J$-modules. 
Since the module $Q$ is injective as a $u_{\zeta}({\mathfrak l}_{J})$-module and 
$E\otimes Q$ is  injective for any $u_\zeta(\mathfrak{l}_J)$-module $E$, the functor
${\mathcal G}$ is an exact functor. However, the Hopf algebra $ U_\zeta(\mathfrak p_J)$ does not necessarily act on $A_i$, thus we cannot assume that $\on{H}^s(A_i, M)$ are $ U_\zeta(\mathfrak p_J)$-modules. 
Thus we let $ \mathcal G'(-)$ be the restriction of $ \mathcal G(-)$ to the category of (integrable of type 1) $ U_\zeta(\mathfrak b) $-modules to the category of rational  $B$-modules. 

 We first note that for each $ 0 \leq i \leq N^J$, each $\on{H}^s(A_i, M)$ is a 
$ U_\zeta (\mathfrak b)$-module (cf. \ref{sec:2.4}) which is $ u_\zeta(\mathfrak{u}_J)$-trivial. 
Thus we can apply the functor $ \mathcal{G'}(-)$ to get a $ B$-module 
${\mathcal G'}(\text{H}^{s}(A_i, M))$. 
 
We will first prove by induction on $i$ that for $0 \leq i \leq N^J$ as 
$\mathcal{U}(\mathfrak{b})$-modules:
\begin{equation}\label{eq:3.1.1}
{\mathcal G'}(\text{H}^{s}(A_i, M)) \cong
\begin{cases}
S^{r}(V_i) &\text{ if }s = 2r+t,\\
0 &\text{ otherwise}.
\end{cases}
\end{equation}
 For $i=0$ this follows 
from the hypothesis. The statement of this proposition (as $B$-modules) is 
the case when $i=N^J$. 

Assume that \eqref{eq:3.1.1} holds for $i - 1$. We will prove that it is also valid for 
$i$.  By using the PBW basis Lusztig constructed for $ \BU_q(\mathfrak{u}_{\empty})$, one can construct a PBW basis for $ \mathcal{U}_\zeta(\mathfrak{u}_{J})$ by removing the denominators in the divided powers.  Then by using  induction on $i$, one can show that 
$ A_{i-1}$ is a projective $B_i$-module (see Section~\ref{sec:3.1} for the definition of $B_i$). Now we can consider the Lyndon-Hochschild-Serre spectral sequence \cite[Thm. 6.1]{CE}: 
$$
E_{2}^{a,b} = \text{H}^a(A_i,\text{H}^b(B_i,M)) \Rightarrow
\text{H}^{a + b}(A_{i-1},M).
$$
Since $Z_{J}^+.M=0$, then $B_i$ acts on $M$ trivially and $ \on{H}^*(B_i, M)=M\otimes  \on{H}^*(B_i, \bfk)$ by Proposition~\ref{tensor-identity}. The algebra $A_{i}$ acts trivially on $B_{i}$ via 
the adjoint action induced from the right adjoint action of the Hopf algebra 
$\mathcal{U}_\zeta(\mathfrak{p}_J)$. In particular, $ u_\zeta({\mathfrak u}_{J})$ acts on $ B_i$ trivially. 
Thus $A_i=\mathcal{U}_\zeta(\mathfrak{u}_J)/\!\!/ Z_i$ 
acts trivially on $ \on{H}^*(B_i, \BC)$ using the argument following 
Proposition~\ref{tensor-identity}.  The $A_i$-action on  $ \on{H}^*(B_i, M)=M\otimes  \on{H}^*(B_i, \BC)$ is via the action on $M$. Thus, the spectral sequence can be rewritten as 
\begin{equation}\label{eq:3.1.2}
E_{2}^{a,b} = \operatorname{H}^a(A_i,M)\otimes\operatorname{H}^b(B_i,{\bfk})
\Rightarrow \operatorname{H}^{a + b}(A_{i-1},M).
\end{equation}

The cohomology of $B_{i}$ is an exterior algebra with one generator in degree 1 and,  as 
$U_\zeta(\mathfrak{b})$-modules, we have
$$
\operatorname{H}^b(B_i,{\bfk}) =
\begin{cases}
{\bfk} &\text{ if } b = 0,\\
\bfk_{\gamma_{i}} &\text { if } b = 1,\\
0 &\text{ otherwise}. 
\end{cases}
$$

Since the functor ${\mathcal G'}(-)$ is exact, \delete{\red{can this functor be applied here? Need $ u_\zeta (l_J)$-module structure}} we can apply this functor to the above spectral sequence and obtain a new spectral
sequence using the fact that  $u_{\zeta}({\mathfrak l}_{J})$-action on the cohomology   of 
$B_{i}$ is trivial to get the following spectral sequence:
$$
E_{2}^{a,b} = {\mathcal G'}(\operatorname{H}^a(A_i, M)) \otimes\operatorname{H}^b(B_i,{\bfk}) \Rightarrow
    {\mathcal G'} (\operatorname{H}^{a+b}(A_{i-1},M)). 
$$
The spectral sequence consists of at most  two non-zero rows (i.e., when $b=0,1$). 
Therefore, the spectral sequence collapses after the $E_{3}$-page and the abutment can be obtained by calculating the 
differential $d_{2}$. Thus, we have a short exact sequence 
\begin{equation} \label{abutment} 0 \rightarrow E_3^{u-1,1}\rightarrow  \mathcal{G'}(\on{H}^u(A_{i-1},M)) \rightarrow E_3^{u,0} \rightarrow 0
\end{equation}
for all $ u$. On the other hand, we have 
\begin{equation} \label{E3-page} E_3^{u,1}=\ker(d_2: E_2^{u,1}\rightarrow E_2^{u+2,0}), \qquad E_3^{u,0}=\on{coker}(d_2: E_2^{u-2,1}\rightarrow E_2^{u,0}). 
\end{equation}
 Observe that from the induction hypothesis the abutment is 
\begin{equation}\label{assume:i-1}
{\mathcal G'}(\operatorname{H}^{a + b}(A_{i-1},M)) \cong
\begin{cases}
S^{r}(V_{i-1}), &\text{ if } a + b = 2r+t ,\\
0, &\text{ otherwise}.
\end{cases}
\end{equation}
 The first row $E_2^{a,0} = {\mathcal G'}(\operatorname{H}^a(A_i, M))$ is 
what we are trying to determine. Note also that 
$E_2^{a,1} \cong E_{2}^{a,0}\otimes \BC\gamma_{i}$. 

We observe that $E_{2}^{u,0}\cong E_{2}^{u,1}=0$ for $u<t$. Note that 
${\mathcal G'} (\operatorname{H}^{u}(A_{i-1},M))=0$ for $u<t$. 
Therefore, $E_{3}^{u,0}=E_{\infty}^{u,0}=0$ and $E_{3}^{u,1}=E_{\infty}^{u,1}=0$ for $u+1<t$.
Thus, for $u<t$, $d_n^{u-2,1}$ is an isomorphism by \eqref{E3-page} and we have
$$E_{2}^{u,0}\stackrel{d_2}{\cong} E_{2}^{u-2,1}\cong E_{2}^{u-2,0}\otimes \BC \gamma_i$$  
One can conclude inductively that $E_{2}^{u,0}=0$ for $u< t$. 

When $a-t $ is odd, then by \eqref{assume:i-1} and \eqref{abutment}, we have
$E_3^{a,0}=0=E_3^{a-1,1}$ and that $ d_2^{a-2,1}$ is surjective and $ d_2^{a-1,1}$ is injective.  
If $ a=t+1$, the surjectivity of $d_2^{t-1, 1}$ and the fact that $ E_2^{t-1, 1}=0$ implies that $ E_2^{t+1, 0}=0$, which further implies $ E_2^{t+1, 1}=0$.  

We now assume that we have proved that $ E_2^{a, 0}=0$  for an $a$ 
with $ a-t$ odd. Then $ E_2^{a,1}=0$. 
Since $\mathcal{G'}(\on{H}^{a+2}(A_{i-1}, M))=0$, 
we have $E_3^{a+2,0}=0$. Similarly, the surjectivity 
of $d_2^{a,1}: E_2^{a,1}\rightarrow E_{2}^{a+2,0}$
 implies that $ E_2^{a+2,0}=0$. Therefore,  
 by induction on $a$, we conclude that
  $ E_2^{a,0}=0=E_2^{a,1}$ whenever  $a-t$ is odd.  In particular we have $ E_3^{a, 1}=0$ for $a$ with $ a-t$ odd. Thus we have $ E_3^{a, 1}=0 $ for all $a$ by  \eqref{abutment} and \eqref{assume:i-1}. Consequently, we have $ E_3^{a,0}=\mathcal{G'}(\on{H}^a(A_{i-1}, M))$ for all $a$. 

Next we analyze the case when $a-t$ is even. First observe that 
 $d_2: E_2^{a,1} \to E_2^{a + 2,0}$ is a monomorphism, and 
for $a-t$ even, 
\[
E_2^{a,0}/E_2^{a - 2,1} =E_3^{a,0}=  {\mathcal G'}(\operatorname{H}^{a}(A_{i-1},M))\]
This yields a short exact sequence of $U_{\zeta}(\mathfrak b)$-modules:
\begin{equation} \label{symmetricalgses}
0 \to {\mathcal G'} (\operatorname{H}^{a - 2}(A_i, M))\otimes (\gamma_i) \to {\mathcal
G'}(\operatorname{H}^{a}(A_i,M)) \to
S^{(a-t)/2}(V_{i-1}) \to 0.
\end{equation}

We can now use induction on $a-t$ (even) to determine 
${\mathcal G'}(\operatorname{H}^{a-t}(A_i,M))$.
For even $a-t$, assume that
$$
{\mathcal G'}(\operatorname{H}^{a - 2}(A_i, M)) =
S^{(a-t - 2)/2}(V_i).
$$
Then the short exact sequence ~\eqref{symmetricalgses} can be written as 
$$
0 \to S^{(a-t - 2)/2}(V_i)\otimes (\gamma_{i}) \to
{\mathcal G'}(\operatorname{H}^{a}(A_i, M)) \to
S^{(a-t)/2}(V_{i-1})\to 0.
$$
Now set $a = 2r+t$, thus we have isomorphism of  $U_{\zeta}(\mathfrak b)$-modules
\begin{equation}\label{eq:3.2.7}
{\mathcal G'}(\operatorname{H}^{a}(A_i, M)) \cong
(S^{r-1}(V_i)\otimes \bfk\gamma_i)\oplus
S^{r}(V_{i-1}) \cong S^{r}(V_i).
\end{equation} 
Note that the above isomorphism is as $B$-modules. We thus have proved \eqref{eq:3.1.1}, and therefore  statement (a) with the isomorphism being as $ B$-modules. 

Next we need to verify that the identifications given in the statement (a) holds as ${\mathcal U}({\mathfrak p}_{J})$-modules. Recall $Z_{J}$ is a central subalgebra of $ \mathcal{U}_\zeta(\mathfrak{u}_J)$ which is  $ U_\zeta(\mathfrak{p}_J)$-stable under the right adjoint action on $ \mathcal{U}_\zeta(\mathfrak{u}_J)$ (cf. \cite[Cor. 2.7.4]{BNPP}). 
We have 
a spectral sequence of $U_{\zeta}({\mathfrak p}_{J})$-modules: 
$$
E_2^{a,b} = \operatorname{H}^a(u_{\zeta}({\mathfrak u}_{J}),\operatorname{H}^b(Z_J,M))
\Rightarrow \operatorname{H}^{a + b}(\mathcal{U}_{\zeta}({\mathfrak u}_{J}),M)
$$
The algebra $u_{\zeta}({\mathfrak p}_{J})$-action (the restriction 
from the $U_\zeta(\mathfrak{p}_J)$-action which is the same as that  factored through the adjoint action of $ \mathcal U_\zeta (\mathfrak p_J)$)  on $Z_J$ is trivial on since it $Z_J$ is central in $ \mathcal U_\zeta(\mathfrak g)$. Thus  $u_{\zeta}({\mathfrak p}_{J})$ also acts trivially on $\on{H}^b(Z_J,{\bfk})$.  
Also $Z_{J}$ (as augmented algebra) acts trivially on 
 the $ U_\zeta(\mathfrak{p}_J)$-module $M$. Therefore, the above spectral sequence becomes
$$
E_2^{a,b} = \operatorname{H}^a(u_{\zeta}({\mathfrak u}_{J}),M)\otimes \operatorname{H}^b(Z_J,{\bfk}) \Rightarrow
 \operatorname{H}^{a +b}(\mathcal{U}_{\zeta}({\mathfrak u}_{J}),M).
$$
The restriction of the $U_\zeta(\mathfrak{p}_J)$-action to the $u_\zeta(\mathfrak{u}_J)$-action is trivial on each term of the spectral sequence by Proposition~\ref{tensor-identity}. Now we apply the functor ${\mathcal G}(-)$ to obtain a new
spectral sequence of $U_{\zeta}({\mathfrak p}_{J})$-modules:
\begin{equation}\label{ZJ-spec-seq}
E_2^{a,b} = {\mathcal G}(\operatorname{H}^a(u_{\zeta}({\mathfrak u}_{J}),M)) \otimes \operatorname{H}^b(Z_J,{\mathbb C}) \Rightarrow
\mathcal{G}(\operatorname{H}^{a + b}(\mathcal{U}_{\zeta}({\mathfrak u}_{J}),M)).
\end{equation}
Observe that $U_{\zeta}({\mathfrak p}_{J})$ acts on the terms in the spectral sequence with both $u_{\zeta}(\mathfrak{l}_J)$ and $ u_\zeta(\mathfrak{u}_J)$ acting trivially on each term.  Thus 
this is a spectral sequence of $\mathcal U({\mathfrak p}_{J})$-modules. Since we are working in the category of  integrable $U_\zeta(\mathfrak p_J)$-modules of type $ \mathsf 1$, the terms of the spectral sequence are rational $P_J$-modules as well.   By the assumption,  the abutment $\mathcal{G}(\operatorname{H}^{a + b}(\mathcal{U}_{\zeta}({\mathfrak u}_{J}),M))$ of this spectral sequence is nonzero only when 
$a+ b = t$ in which case it is the trivial module $\bfk$. Moreover, by \eqref{eq:3.2.7}, as ${\mathcal U}({\mathfrak b})$-modules, 
\begin{equation}\label{u-cohomology}
{\mathcal G}(\operatorname{H}^a(u_{\zeta}({\mathfrak u}_{J}),M)) ={\mathcal G'}(\operatorname{H}^a(u_{\zeta}({\mathfrak u}_{J}),M)) \cong S^\frac{a-t}{2}({\mathfrak u}_{J}^*).
\end{equation}

We note that both are $\mathcal U(\mathfrak p_J)$-modules. They are actually rational $P_J$-modules. The isomorphism is as $ B$-modules. Proposition~\ref{module-isomorphism} implies that ${\mathcal G}(\operatorname{H}^a(u_{\zeta}({\mathfrak u}_{J}),M))\cong  S^\frac{a-t}{2}({\mathfrak u}_{J}^*)$ as $P_J$-modules. 

(b) Apply the Lyndon-Hochschild-Serre (LHS) spectral sequence 
to $u_{\zeta}({\mathfrak u}_{J})$ as a normal Hopf subalgebra in $u_{\zeta}({\mathfrak p}_{J})$:
\begin{equation*} 
E_{2}^{i,j}=\operatorname{H}^{i}(u_{\zeta}({\mathfrak l}_{J}),
\operatorname{H}^{j}(u_{\zeta}({\mathfrak u}_{J}),M)\otimes Q)\Rightarrow 
\operatorname{H}^{i+j}(u_{\zeta}({\mathfrak p}_{J}),M\otimes Q).
\end{equation*} 
Since $Q$ is injective over $u_{\zeta}({\mathfrak l}_{J})$ (thus tensor product with any $u_{\zeta}({\mathfrak l}_{J})$-module remains injective),  this spectral sequence collapses and yields: 
\begin{equation*} 
\text{H}^{n}(u_{\zeta}({\mathfrak p}_{J}),M\otimes Q)=\text{Hom}_{u_{\zeta}({\mathfrak l}_{J})}
(\bfk,\operatorname{H}^{n}(u_{\zeta}({\mathfrak u}_{J}),M)\otimes Q). 
\end{equation*}   
The result now follows by part (a) .
\end{proof}

\subsection{} In the next theorem we identify important axiomatic conditions which allow one to realize a $G$-module isomorphism between 
$\text{ind}_{P_{J}}^{G} (S^{\bullet}({\mathfrak u}_{J}^{*})\otimes \bfk_\gamma)$ and cohomology for modules over the small quantum group in order to 
answer Question~\ref{realizequestion2}. 

\begin{theorem}\label{thm:quantumrealization} Let $J\subseteq \Delta$, $M$ be a $U_{\zeta}({\mathfrak g})$-module with $Z_{J}^+.M=0$ and $Q$ be a $U_{\zeta}({\mathfrak p}_{J})$-module such that 
$Q$ is  injective as $u_{\zeta}({\mathfrak l}_{J})$-module and  trivial as $U_{\zeta}({\mathfrak u}_{J})$-module. Assume that the following 
two conditions hold: 
\begin{itemize} 
\item[(i)] As a rational $B$-module, and for some $t\geq 0$,  
\begin{equation}
\operatorname{Hom}_{u_{\zeta}({\mathfrak l}_{J})}({\bfk},\operatorname{H}^{n}({\mathcal U}_{\zeta}({\mathfrak u}_{J}), M)\otimes Q)\cong
\begin{cases} \bfk & n=t, \\
0 & \text{otherwise.} 
\end{cases} 
\end{equation} 
\item[(ii)] For $\gamma\in X_{P_{J}}\cap X^{+}$ we have 
$R^{i}\operatorname{ind}_{U_{\zeta}({\mathfrak p}_{J})}^{U_{\zeta}({\mathfrak g})} (Q \otimes \bfk_{\ell\gamma}) =0$ for $i>0$. 
\end{itemize} 
Then there exists an isomorphism of rational $G$-modules:
\begin{equation}
\operatorname{H}^{n}(u_{\zeta}({\mathfrak g}),M\otimes \operatorname{ind}_{U_{\zeta}({\mathfrak p}_{J})}^{U_{\zeta}({\mathfrak g})} (Q\otimes \bfk_{\ell\gamma}))\cong
\begin{cases} \operatorname{ind}_{P_{J}}^{G} (S^{\frac{n-t}{2}}({\mathfrak u}^{*}_{J}) \otimes \bfk_{\gamma}) & n-t \equiv 0 \ (\operatorname{mod }2), \\
0 & \text{otherwise}. 
\end{cases} 
\end{equation} 
\end{theorem}

\begin{proof} We first consider the functors (cf. \cite[I 6.12]{Jan1})
\[
{\mathcal F}_{1}(-)=\text{Hom}_{u_{\zeta}({\mathfrak g})}({\bfk},\text{ind}_{U_{\zeta} ({\mathfrak p}_{J})}^{U_{\zeta}({\mathfrak
g})}(-)) \text{ and } {\mathcal F}_{2}(-)=\text{ind}_{P_{J}}^{G}
\text{Hom}_{u_{\zeta}({\mathfrak p}_{J})}({\bfk},-)
\]
with ${\mathcal F}_1, {\mathcal F}_2$ : $U_{\zeta}({\mathfrak
p}_J)$-mod $\to$ $U_{\zeta}({\mathfrak g})$-mod (or $G$-mod). Observe that we are 
using the Frobenius map and the following identification of functors:
$$\text{ind}_{P_{J}}^{G}(-)\cong
\text{ind}_{U_{\zeta}({\mathfrak p}_{J})/\!\!/u_{\zeta}({\mathfrak
p}_{J})}^{U_{\zeta}({\mathfrak g})/\!\!/u_{\zeta}({\mathfrak g})}(-).
$$

Set $D=M\otimes Q\otimes \bfk_{\ell\gamma}$.
The functors ${\mathcal F}_1$ and ${\mathcal F}_2$ are naturally
isomorphic and there exist two spectral sequences:
\begin{equation*}\label{firstcohoeqn1}
{}'\!E_{2}^{i,j}=\HH^{i}(u_{\zeta}({\mathfrak
g}),R^{j}\text{ind}_{U_{\zeta} ({\mathfrak
p}_{J})}^{U_{\zeta}({\mathfrak g})}D)\Rightarrow
(R^{i+j}{\mathcal F}_{1})(D) \quad \text{ and }
\end{equation*}
\begin{equation}\label{firstcohoeqn2}
E_{2}^{i,j}=R^{i}\text{ind}_{P_{J}}^{G}
\HH^{j}(u_{\zeta}({\mathfrak p}_{J}),D) \Rightarrow (R^{i+j}{\mathcal
F}_{2})(D)
\end{equation}
which converge to the same abutment. However, by the tensor identity and condition (ii), 
\begin{equation*}
R^{i}\text{ind}_{U_{\zeta} ({\mathfrak p}_{J})}^{U_{\zeta}({\mathfrak g})} D \cong M\otimes 
R^{i}\text{ind}_{U_{\zeta} ({\mathfrak p}_{J})}^{U_{\zeta}({\mathfrak g})}[Q \otimes \bfk_{\ell\gamma}] \cong 0
\end{equation*} 
for $i>0$. Consequently, the first spectral sequence collapses and we can combine this with 
the second spectral sequence to obtain a first quadrant spectral sequence: 
\begin{equation} \label{spectralseq11} 
E_{2}^{i,j}=R^{i}\text{ind}_{P_{J}}^{G} (\operatorname{H}^{j}(u_{\zeta}({\mathfrak p}_{J}), M\otimes  Q )\otimes \bfk_{\gamma}) \Rightarrow 
\operatorname{H}^{i+j}(u_{\zeta}({\mathfrak g}),M\otimes \operatorname{ind}_{U_{\zeta}({\mathfrak p}_{J})}^{U_{\zeta}({\mathfrak g})} (Q\otimes \bfk_{\ell\gamma}) ). 
\end{equation} 
One can rewrite the spectral sequence (\ref{spectralseq11}) using Proposition~\ref{quantumExtcalculation}(b) as 
\begin{equation} \label{spectralseq12}
E_{2}^{i,j}=R^{i}\text{ind}_{P_{J}}^{G} (S^{\frac{j-t}{2}}({\mathfrak u}^{*}_{J}) \otimes \bfk_{\gamma}) \Rightarrow 
\operatorname{H}^{i+j}(u_{\zeta}({\mathfrak g}),M \otimes  \operatorname{ind}_{U_{\zeta}({\mathfrak p}_{J})}^{U_{\zeta}({\mathfrak g})} (Q\otimes \bfk_{\ell\gamma})). 
\end{equation}
Since $\gamma \in X_{P_{J}}\cap X^{+}$, the Grauert-Riemenschneider vanishing theorem implies that 
$$R^{i}\text{ind}_{P_{J}}^{G} (S^{\bullet}({\mathfrak u}^{*}_{J})\otimes \bfk_{\gamma})=0$$ 
for $i>0$. The spectral sequence (\ref{spectralseq12}) collapses and yields the isomorphism stated in the theorem. 
\end{proof} 
\section{Non-shifted realization of the coordinate ring of the twisted cotangent bundle} 

\subsection{} Let $J\subseteq \Delta$ and $X_{J}^{+}$ be the $J$-dominant weights 
(i.e., $\langle \mu,\alpha^{\vee} \rangle \geq 0$ for all $\alpha\in J$). 
For $\mu\in X_{J}^{+}$, let 
$$H^{0}_{J}(\mu)=
\operatorname{ind}_{U_{\zeta}({\mathfrak b}_{J})}^{U_{\zeta}({\mathfrak l}_{J})} \bfk_\mu.$$ 
In particular when $J=\varnothing$, we obtain the ordinary induced module 
$H^{0}(\mu)=\operatorname{ind}_{U_{\zeta}({\mathfrak b})}^{U_{\zeta}({\mathfrak g})} \bfk_\mu$
where $\mu\in X^{+}$. For $\mu\in X_{J}^{+}$, let $L_{J}(\mu)$ be the finite dimensional simple 
$U_{\zeta}({\mathfrak l}_{J})$-module which appears in the socle of $H^{0}_{J}(\mu)$. 

Next we observe that, for any $w\in {}^J W$,  the weights $w\cdot 0$ and $-w_{0,J}(w\cdot 0)$ are in $X_{J}^{+}$, and 
in fact in the bottom alcove of $X_{J}^{+}$ when $l>h$ \cite[Prop. 3.6.1]{UGAVIGRE1}. Therefore, 
$$L_{J}(w\cdot 0)\cong H^{0}_{J}(w\cdot 0)$$
and 
$$L_{J}(w\cdot 0)^{*}=L_{J}(-w_{0,J}(w\cdot 0))=H^{0}_{J}(-w_{0,J}(w\cdot 0)).$$  

We will also be using the fact that 
$ w\cdot 0=-\sum_{\alpha \in \Phi^+(w)}\alpha$. 
Here $\Phi^+(w)=\{ \alpha \in \Phi^+\;|\; w^{-1}(\alpha)\in \Phi^-\}$. If $w\in {}^JW$, 
then $ \Phi^+(w)\subseteq \Phi^+\setminus \Phi^+_J$. If $w\in W_J$ then $\Phi^+(w)\subseteq \Phi_J$ and $w(\Phi_J)\subseteq \Phi_J$.

Let $\epsilon_{J}:=(\ell-1)\sum_{\alpha\in J}\omega_{\alpha}$. In this 
case $H^{0}_{J}(\epsilon_J)=L_{J}(\epsilon_J)$ 
is the Steinberg representation for $U_{\zeta}({\mathfrak l}_{J})$. 
We also note that $H^{0}_{J}(\epsilon_J)$ 
is the unique  composition factor of  
$ H^0(\epsilon_J)$ as 
$U_\zeta(\mathfrak l_J)$-module 
with highest weight $\epsilon_J$. In fact 
\[  H^0(\epsilon_J)|_{U_\zeta(\mathfrak l_J)}\cong H^0_J(\epsilon_J)\oplus E
\]
with $E$ being a $U_\zeta(\mathfrak l_J)$-module without weight $ \epsilon_J$. In particular, 
$$\hom_{U_\zeta(\mathfrak l_J)}(H^0(\epsilon_J), H^0_J(\epsilon_1))\cong \bfk.$$

\begin{lemma} \label{coh-direct-summand} Assume $\ell\geq h-1$. Then for any $U_\zeta(\mathfrak p_J)$-module $M$ and any $U_\zeta(\mathfrak l_J)$-module $Q$ such that $Q$ is injective as $ u_\zeta(\mathfrak l_J)$-module, $\on{H}^{n}({\mathcal U}_{\zeta}({\mathfrak u}_{J}), M)\otimes Q$ is isomorphic, as a $u_\zeta(\mathfrak l_J)$-module, to a direct summand of 
\[\bigoplus_{y\in{}^JW,\; l(y)=n}M\otimes L_J(y\cdot 0)^*\otimes Q. 
\]
In particular, for any $u_\zeta(\mathfrak l_J)$-module $E$, $\on{Hom}_{u_\zeta(\mathfrak l_J)}(E, \on{H}^n(\mathcal U_\zeta(\mathfrak u_J), M)\otimes Q)=0 $ 
if 
$$
\on{Hom}_{u_\zeta(\mathfrak l_J)}(E,  M\otimes L_J(y\cdot 0)\otimes Q)=0$$
for all $ y\in {}^JW$ with $ l(y)=n$. Similar statements can be 
formulated for the functor $ \on{Hom}_{u_\zeta(\mathfrak l_J)}(-, E)$.
\end{lemma}

\begin{proof}
We can first apply \cite[Theorem 3.5.2]{UGAVIGRE2} to deduce that $\on{H}^{n}({\mathcal U}_{\zeta}({\mathfrak u}_{J}), 
M)$ is a $u_{\zeta}({\mathfrak l}_{J})$-subquotient of 
$$M\otimes \operatorname{H}^{n}({\mathcal U}_{\zeta}({\mathfrak u}_{J}), {\bfk}).$$   
According to \cite[Theorem 6.4.1]{UGAVIGRE2} since $\ell\geq h-1$ we have  
$$\text{H}^{n}({\mathcal U}_{\zeta}({\mathfrak u}_{J}), {\bfk})\cong \bigoplus_{y\in ^JW,\ l(y)=n} L_{J}(-w_{0,J}(y\cdot 0)).$$ 
Since $ Q$ is $u_\zeta(\mathfrak l_J)$-injective, any short exact sequence of $u_\zeta(\mathfrak l_J)$-modules tensored by $Q$ becomes split as $u_\zeta(\mathfrak l_J)$-modules.  Then $\on{H}^n(\mathcal U_\zeta(\mathfrak u_J, M)\otimes Q$ is a direct $u_\zeta(\mathfrak l_J)$-direct summand of $M\otimes \on{H}^n(\mathcal U_\zeta(\mathfrak u_J, 
\bfk)\otimes Q$. Now the second part of the lemma follows by applying the functor additive functors $ \on{Hom}_{u_\zeta(\mathfrak l_J)}(E, -)$ and $ \on{Hom}_{u_\zeta(\mathfrak l_J)}(-, E)$. 
\end{proof}
\subsection{} We can now apply the theory developed in Section 3 in order to give a precise realization of 
the coordinate algebra of the twisted cotangent bundle. 
\vskip .15cm 
\noindent 
Let {\it $M=H^{0}(\epsilon_J)^{*}$ (a $U_\zeta(\mathfrak g)$-module) and $Q_w=H^{0}_{J}(\epsilon_J)\otimes H^{0}_{J}(-w_{0,J}(w\cdot 0))$ for a fixed $w\in{}^JW$.}
 
\begin{theorem}\label{thm:quantumnonshiftrealization} Let $\ell>2h-1$, $J\subseteq \Delta$ and $\epsilon_J=\sum_{\alpha\in J}(\ell-1)\omega_{\alpha}$ with $w\in{} ^JW$ and $\gamma\in X_{P_{J}}\cap X^{+}$. Set 
$$N_w:=H^{0}(\epsilon_J)^{*}\otimes  \operatorname{ind}_{U_{\zeta}({\mathfrak p}_{J})}^{U_{\zeta}({\mathfrak g})}[ H^{0}_{J}(\epsilon_J)\otimes 
H^{0}_{J}(-w_{0,J}(w\cdot 0))\otimes \bfk_{\ell\gamma}],$$
and assume that 
\begin{itemize}
\item[(a)] $\operatorname{Hom}_{u_{\zeta}({\mathfrak l}_{J})}({\bfk},\operatorname{H}^{0}({\mathcal U}_{\zeta}({\mathfrak u}_{J}), M)\otimes Q_w)\cong \bfk$,
\item[(b)] $\epsilon_J-w_{0,J}(w\cdot 0)+\ell\gamma \in X^{+}$.  
\end{itemize} 
Then there exists an isomorphism of rational $G$-modules:
\begin{equation}\label{isocohoN}
\operatorname{H}^{n}(u_{\zeta}({\mathfrak g}),N_w)\cong
\begin{cases} \operatorname{ind}_{P_{J}}^{G} S^{\frac{n}{2}}({\mathfrak u}^{*}_{J}) \otimes \bfk_\gamma & \ n \equiv 0 \ (\operatorname{mod }2) ,  \\
0 & \text{otherwise} .
\end{cases} 
\end{equation} 
Furthermore, in the case when $w=\operatorname{id}$ for $N_w$, (\ref{isocohoN}) holds for $\ell>h$ with the condition (b). The condtion (a) is satisfied in the case when $w=\operatorname{id}$. 
\end{theorem}

\begin{proof}  The module $Q_w$ is a $U_\zeta(\mathfrak l_J)$-module inflated to a $U_\zeta(\mathfrak p_J)$.  
 Since $L_{J}(\epsilon_J)=H^{0}_{J}(\epsilon_J)$ is an injective $u_{\zeta}({\mathfrak l}_{J})$-module, thus $Q_w$ is an injective $u_{\zeta}({\mathfrak l}_{J})$-module and also injective $U_\zeta(\mathfrak l_J)$. 
 Therefore,  $M$ and $Q_{w}$ satisfy the conditions of Theorem~\ref{thm:quantumrealization}. 
 Since $\epsilon_J$ is an $\ell$-restricted weight, $M$ is a quotient of the baby Verma module $u_{\zeta}({\mathfrak g})\otimes_{u_{\zeta}({\mathfrak b})}\bfk_{-w_{0}\epsilon_J}$, thus $Z_{J}^+.M=0$ when $M$ considered a $\mathcal U_\zeta(\mathfrak g)$-module. 
We need to verify conditions (i) and (ii) of Theorem~\ref{thm:quantumrealization} in order to conclude our result. 

For condition (i), we need to show that as a rational $B$-module, 
\begin{equation}\label{condition-i}
\on{Hom}_{u_{\zeta}({\mathfrak l}_{J})}({\bfk},\on{H}^{n}({\mathcal U}_{\zeta}({\mathfrak u}_{J}), M)\otimes Q_w)\cong
\begin{cases} \bfk & n=0, \\
0 & \text{otherwise.} 
\end{cases} 
\end{equation} 
Assume the left hand side of \eqref{condition-i} 
\begin{equation}
\operatorname{Hom}_{u_{\zeta}({\mathfrak l}_{J})}({\bfk},\operatorname{H}^{n}({\mathcal U}_{\zeta}({\mathfrak u}_{J}),H^{0}(\epsilon_J)^{*})\otimes L_{J}(\epsilon_J)\otimes  L_{J}(-w_{0,J}(w\cdot 0)))\neq 0. 
\end{equation} 
 Then,  by Lemma~\ref{coh-direct-summand}, there is  $y\in {}^JW$ with $l(y)=n$ such that
\begin{equation*}
\operatorname{Hom}_{u_{\zeta}({\mathfrak l}_{J})}({\bfk},H^{0}(\epsilon_J)^{*}\otimes L_{J}(-w_{0,J}(y\cdot 0))\otimes L_{J}(\epsilon_J)\otimes L_{J}(-w_{0,J}(w\cdot 0)))\neq 0
\end{equation*} 
which is equivalent to 
\begin{equation*}
\operatorname{Hom}_{u_{\zeta}({\mathfrak l}_{J})}(L_{J}(-w_{0,J}\epsilon_J),H^{0}(\epsilon_J)^{*}\otimes L_{J}(-w_{0,J}(y\cdot 0))\otimes L_{J}(-w_{0,J}(w\cdot 0)))\neq 0.
 \end{equation*} 
All the modules under consideration are $U_{\zeta}({\mathfrak l}_{J})$-modules, so we can conclude that 
\begin{equation}
\operatorname{Hom}_{U_{\zeta}({\mathfrak l}_{J})}(L_{J}(-w_{0,J}\epsilon_J)\otimes L(\nu)^{[1]},H^{0}(\epsilon_J)^{*} \otimes  L_{J}(-w_{0,J}(y\cdot 0))\otimes L_{J}(-w_{0,J}(w\cdot 0) ))\neq 0 
\end{equation} 
for some $\nu\in X_{J}^{+}$. Since  
$$L_{J}(-w_{0,J}(y\cdot 0))\cong H^{0}_{J}(-w_{0,J}(y\cdot 0))\cong \text{ind}_{U_{\zeta}({\mathfrak b}_{J})}^{U_{\zeta}({\mathfrak l}_{J})} \bfk_{-w_{0,J}(y\cdot 0)}$$
it follows by Frobenius reciprocity that 
\begin{equation} \label{bJ-equation} 
\operatorname{Hom}_{U_{\zeta}({\mathfrak b}_{J})}(L_{J}(-w_{0,J}\epsilon_J)\otimes L_J(\nu)^{[1]},H^{0}(\epsilon_J)^{*}\otimes   \bfk_{-w_{0,J}(y\cdot 0)}\otimes L_{J}(-w_{0,J}(w\cdot 0)))\neq 0 .
\end{equation} 
The head of $L_{J}(-w_{0,J}\epsilon_J)\otimes L(\nu)^{[1]}$ as a $U_{\zeta}({\mathfrak b}_{J})$-module is one-dimensional and isomorphic to 
$-w_{0,J}\epsilon_J+l\nu$. We can now deduce from (\ref{bJ-equation}) that 
\begin{equation} \label{weighteq1} 
-w_{0,J}\epsilon_J+\ell \nu=\mu-w_{0,J}(y\cdot 0)+\sigma
\end{equation} 
for some $\nu\in X_{J}^{+}$, some  weight 
$\mu$ of $H^{0}(\epsilon_J)^{*}$,  and  
a weight $\sigma$ of $L_{J}(-w_{0,J}(w\cdot 0))=L_J(w\cdot 0)^*$. 
We note that for any $U_\zeta(\mathfrak l_J)$-module 
$V$, $ \mu$ is a weight of $V$ if any only if 
$ -w(\mu)$ is a weight of $V^*$ for any $w\in W_J$. 
Applying $-w_{0,J}$ to (\ref{weighteq1}) justifies revising the weight condition as follows: 
\begin{equation} \label{weighteq2} 
\epsilon_J+\ell\nu=\mu+y\cdot 0 +\sigma
\end{equation} 
where $\nu\in X_{J}^{+}$, $\mu$ an weight of $H^{0}(\epsilon_J)$ and $\sigma$ a weight of $L_{J}(w\cdot 0)$. 

The Weyl group for $\Delta-J$, $W_{\Delta-J}$ is generated by the simple reflections $s_{\beta}$ such that $\beta\in \Delta-J$. 
Note that $s_{\beta}(\omega_{\alpha})=\omega_{\alpha}$ when $\beta\in \Delta-{J}$, $\alpha\in J$, thus $\tilde w(\omega_{\alpha})=\omega_{\alpha}$
for all $\tilde w\in W_{\Delta-J}$, $\alpha\in J$. Therefore, we can find some $\widetilde{y}\in W_{\Delta-J}$ such that $\nu^{\prime}=\widetilde{y}\nu\in X^{+}$. Applying 
$\widetilde{y}$ to (\ref{weighteq2}) yields: 
\begin{equation} \label{weighteq3} 
\epsilon_{J}+\ell \nu^{\prime}=\mu^{\prime}+\widetilde{y}(y\cdot 0)+\widetilde{y}\sigma
\end{equation} 
where $\mu^{\prime}$ is a weight of $H^{0}(\epsilon_{J})$. Rewriting this equation, one has 
\begin{equation} \label{weighteq4} 
\epsilon_{J}-\mu^{\prime}+\ell \nu^{\prime}=\widetilde{y}(y\cdot 0)+\widetilde{y}\sigma.
\end{equation} 
Taking the inner product with $\alpha_{0}^{\vee}$ (with $ \alpha_0$ being the highest short root)  and using the facts that 
$\langle \beta,\alpha_{0}^{\vee} \rangle \geq 0$ and $\langle \delta,\alpha_{0}^{\vee} \rangle \leq 2(h-1)$ for all 
$\beta\in \Phi^{+}$ and $\delta\leq z\cdot 0$ where $z\in W$, one has 
\begin{equation} \label{weighteq5} 
0\leq \langle \epsilon_{J}-\mu^{\prime},\alpha_{0}^{\vee} \rangle +\ell \langle \nu^{\prime},\alpha_{0}^{\vee} \rangle \leq 4(h-1) < 2\ell. 
\end{equation} 
The equation (\ref{weighteq5}) implies that $\langle \nu^{\prime},\alpha_{0}^{\vee} \rangle=0,1$. 

From (\ref{weighteq4}) it follows that $\ell\nu^{\prime}$ is in the root lattice by noting that $ \mu-\epsilon_J$ as well as  all weights of $ L_J(y\cdot 0)$ and $L_J(w\cdot 0)$ are  in the root lattice.  Since $\ell>h$, $\ell\nmid |X/{\mathbb Z}\Phi|$, and 
$\nu^{\prime}$ must be in the root lattice 
(i.e., cannot be minuscule). Consequently, 
$\nu=0$. Since $y\cdot 0, \sigma \in -\mathbb N\Phi^+$, \eqref{weighteq2} 
implies that $y=\on{id}$, $\sigma =0$, $n=0$, 
and $\mu=\epsilon_J$. 

Note when $w=\text{id}$ for $N_w$ with $\ell>h$, one has $\sigma=0$, and one can replace (\ref{weighteq5}) with 
\begin{equation} \label{weighteq6} 
0\leq \langle \epsilon_{J}-\mu^{\prime},\alpha_{0}^{\vee} \rangle +\ell \langle \nu^{\prime},\alpha_{0}^{\vee} \rangle \leq 2(h-1)<2\ell. 
\end{equation} 
The arguments from the preceding paragraph can then be repeated to draw the same conclusions (i.e., $\nu=0$ and $y=\text{id}$). Our analysis proves that for $n>0$
$$\operatorname{Hom}_{u_{\zeta}({\mathfrak l}_{J})}({\bfk},\operatorname{H}^{n}({\mathcal U}_{\zeta}({\mathfrak u}_{J}), M)\otimes Q_w)
\cong 0.$$
Now we can use the condition (a) in the statement of the theorem to conclude that (i) holds. 

Next we show that condition (a) holds in the case when $w=\text{id}$. Since  $Q_{\on{id}}=L_J(\epsilon_J)$ is trivial as $u_\zeta(\mathfrak u_J)$-module and $\operatorname{H}^0(\mathcal U_\zeta(\mathfrak u_J), M)=\on{Hom}_{\mathcal U_\zeta(\mathfrak u_J)}(H^0(\epsilon_J), \bfk)$, we have
\begin{eqnarray*}
\operatorname{Hom}_{u_{\zeta}({\mathfrak l}_{J})}({\bfk},\on{H}^{n}({\mathcal U}_{\zeta}({\mathfrak u}_{J}), M)\otimes Q_{\on{id}})
&\cong&\on{Hom}_{u_\zeta(\mathfrak l_J)}(\bfk, L_J(\epsilon_J)\otimes \on{Hom}_{\mathcal U_\zeta(\mathfrak u_J)}(H^0(\epsilon_J), \bfk))\\ 
&\cong&\on{Hom}_{u_\zeta(\mathfrak l_J)}(\bfk,  \on{Hom}_{u_\zeta(\mathfrak u_J)}(H^0(\epsilon_J), L_J(\epsilon_{J}))\\
&\cong&\on{Hom}_{u_\zeta(\mathfrak p_J)}(H^0(\epsilon_J), L_J(\epsilon_{J})).
\end{eqnarray*}
Set $S:=\on{Hom}_{u_\zeta(\mathfrak p_J)}(H^0(\epsilon_J), L_J(\epsilon_{J}))\cong
\on{Hom}_{u_\zeta(\mathfrak p_J)}(H^0(\epsilon_J), H^{0}_J(\epsilon_{J}))$. 
By using a slight variation of the spectral sequences in the proof of Theorem~\ref{thm:quantumrealization} (cf. \cite[II 12.7(2)]{Jan1}), 
we have 
\begin{equation} \label{eq:S-iso}
S\cong \operatorname{ind}_{B}^{P_{J}}
\on{Hom}_{u_\zeta(\mathfrak b_J)}(H^0(\epsilon_J), \bfk_{\epsilon_{J}}).
\end{equation} 

Now set $V=\on{Hom}_{u_\zeta(\mathfrak b_J)}(H^0(\epsilon_J), \bfk_{\epsilon_{J}})$. The verification of 
\eqref{condition-i} for $ w=\text{id}$ will be done once we show that $V\cong \bfk$ as a $B$-module. This will imply that $S\cong \bfk$ as a $P_{J}$-module using (\ref{eq:S-iso}). 

We will next analyze the structure of $V$ as a 
$B$-module. First, we consider the socle of $V$. Suppose that 
$0\neq \on{Hom}_{B}(\bfk_{ \nu},V)\cong \on{Hom}_{U_\zeta(\mathfrak b)}(H^0(\epsilon_J), \bfk_{\epsilon_{J}+\ell \nu})$.  
By using Frobenius reciprocity, it follows that $\epsilon_{J}+\ell \nu$ is in $X^{+}$.  Since 
$$\on{Hom}_{U_\zeta({\mathfrak b})}(H^0(\epsilon_J), \bfk_{\epsilon_{J}+\ell \nu})\neq 0,$$ 
it follows that $\epsilon_{J}+\ell \nu \leq \epsilon_{J}$, thus $\ell\nu\leq 0$. 
Moreover, $\nu$ has to be dominant because $\epsilon_{J}+\ell \nu$ is 
dominant (use the definition of $\epsilon_{J})$. Therefore, 
$\ell \nu\in {\mathbb N}\Phi^{-}\cap X^{+}=\{0\}$, and $\nu=0$. Furthermore,  
$$\on{Hom}_{B}(\bfk,V)\cong \on{Hom}_{U_\zeta(\mathfrak b)}(H^0(\epsilon_J), \bfk_{\epsilon_{J}})
\cong \on{Hom}_{U_\zeta(\mathfrak g)}(H^0(\epsilon_J), H^{0}(\epsilon_{J}))\cong \bfk.$$ 
This shows that $\text{soc}_{B}V\cong \bfk$. 

The next step is to show that $V\cong \bfk$. First, one has $V\subseteq \on{Hom}_{u_\zeta(\mathfrak t)}(H^0(\epsilon_J), \bfk_{\epsilon_{J}})$. 
By using a weight space decomposition for  $H^0(\epsilon_J)$, one sees that any $B$-composition factor of $V$ 
must be of the form $\bfk_{\nu}$ where 
$\bfk_{\epsilon_{J}-\ell\nu}$ is a $U_\zeta(\mathfrak b)$-composition factor of 
$H^{0}(\epsilon_{J})$. Since $\epsilon_{J}$ is the highest weight of $H^{0}(\epsilon_{J})$, it follows that  $\ell\nu \in {\mathbb N}\Phi^{+}$. 

If $V/\text{soc}_{B}V\neq 0$ then 
there exists a $B$-composition factor 
$\bfk_{\nu}$ of $V$ such that $\text{Ext}^{1}_{B}(\bfk_{\nu},\bfk)
\neq 0$. As a consequence of the Borel-Bott-Weil Theorem (cf. \cite[II 5.5]{Jan1}), 
$\nu=-\alpha$ where $\alpha\in \Delta$. This contradicts the prior paragraph that $\ell \nu\in \mathbb{N}\Phi^+$. Hence, 
$V\cong\bfk$ as $B$-module. 

We will now verify condition (ii) of Theorem~\ref{thm:quantumrealization} by showing that  
\begin{align}\label{condition-ii}
R^{i}\text{ind}_{U_{\zeta}({\mathfrak p}_{J})}^{U_{\zeta}({\mathfrak g})} [H^{0}_{J}(\epsilon_J)\otimes H^{0}_{J}(-w_{0,J}(w\cdot 0))\otimes \bfk_{\ell\gamma}] =0
\end{align}
for $i>0$. Set $\widehat{S}:=H^{0}_{J}(-w_{0,J}(w\cdot 0))$. There exists a spectral sequence 
$$E_{2}^{i,j}
=R^{i}\text{ind}_{U_{\zeta}({\mathfrak p}_{J})}^{U_{\zeta}({\mathfrak g})} 
R^{j}\text{ind}_{U_{\zeta}({\mathfrak b})}^{U_{\zeta}({\mathfrak p}_{J})} 
[\bfk_{\epsilon_J} \otimes \widehat{S}\otimes \bfk_{\ell\gamma}] 
\Rightarrow R^{i+j}\text{ind}_{U_{\zeta}({\mathfrak b})}^{U({\mathfrak g})} 
[\bfk_{\epsilon_J} \otimes \widehat{S}\otimes \bfk_{\ell\gamma}].$$ 
Observe that as a $U_{\zeta}({\mathfrak l}_{J})$-module one has by the generalized tensor identity 
\begin{eqnarray*}
R^{j}\text{ind}_{U_{\zeta}({\mathfrak b})}^{U_{\zeta}({\mathfrak p}_{J})} [\bfk_{\epsilon_J} \otimes \widehat{S}\otimes \bfk_{\ell\gamma}] &\cong & 
R^{j}\text{ind}_{U_{\zeta}({\mathfrak b_{J}})}^{U_{\zeta}({\mathfrak l}_{J})} [H^{0}_{J}(\epsilon_J) \otimes \widehat{S}\otimes \bfk_{\ell\gamma}] \\
& \cong &    
(R^{j}\text{ind}_{U_{\zeta}({\mathfrak b_{J}})}^{U_{\zeta}({\mathfrak l}_{J})} {\bfk} )\otimes  [H^{0}_{J}(\epsilon_J) \otimes \widehat{S}\otimes \bfk_{\ell\gamma}]\\
&=& 0
\end{eqnarray*}
for $j>0$. Therefore, this spectral sequence collapses and gives the following isomorphism for $i\geq 0$: 
$$R^{i}\text{ind}_{U_{\zeta}({\mathfrak p}_{J})}^{U_{\zeta}({\mathfrak g})} [H^{0}_{J}(\epsilon_J)\otimes \widehat{S}\otimes \bfk_{\ell\gamma}] \cong 
R^{i}\text{ind}_{U_{\zeta}({\mathfrak b})}^{U({\mathfrak g})} [\bfk_{\epsilon_J}\otimes \widehat{S}\otimes \bfk_{\ell\gamma}].$$ 

Let $\mu$ be a weight of $\widehat{S}=H^{0}_{J}(-w_{0,J}(w\cdot 0))$. Then \eqref{condition-ii} will follow by Kempf's Vanishing Theorem, if we demonstrate that $\epsilon_J+\mu+l\gamma \in X^{+}$. 
Observe that $\langle \alpha, \beta^{\vee} \rangle \leq 0$ for all $\alpha\in J$ and $\beta\in \Delta-J$. This implies that for $\beta\in \Delta-J$, 
\begin{equation*}
\langle \mu,\beta^{\vee} \rangle \geq \langle -w_{0,J}(w\cdot 0),\beta^{\vee} \rangle 
\end{equation*} 
because $-w_{0,J}(w\cdot 0)-\mu \in {\mathbb N}\Phi_{J}^{+}$. Consequently, for $\beta\in \Delta-J$, 
\begin{equation}\label{dominant1}
\langle \epsilon_J+\mu+l\gamma, \beta^{\vee} \rangle \geq \langle \epsilon_J+-w_{0,J}(w\cdot 0)+\ell\gamma , \beta^{\vee} \rangle \geq 0
\end{equation} 
since $\epsilon_J-w_{0,J}(w\cdot 0)+l\gamma\in X^{+}$. On the other hand, let $J=J_{1}\cup J_{2} \cup \dots \cup J_{t}$ be the decomposition of $J$ corresponding to decomposing $\Phi_{J}$ into 
a union of irreducible root systems. Let $J_{s}$ be one of the components, and $\delta$ be either the highest long root or the highest short root of $\Phi_{J_{s}}$. The lowest 
weight of $L_{J}(-w_{0,J}(w\cdot 0))$ is $-w\cdot 0$, and since $w\in {}^JW$, 
$$\langle -w\cdot 0,\delta^{\vee} \rangle =-\langle \rho,w^{-1}\delta^{\vee}\rangle +\langle \rho,\delta^{\vee} \rangle\geq -(h-1).$$ 
Now $\mu-(-w\cdot 0) \in {\mathbb N}\Phi^{+}_{J}$ so $\langle \mu-(-w\cdot 0), \delta^{\vee}  \rangle\geq 0$, thus $\langle \mu, \delta^{\vee} \rangle > -(h-1)$. 
If $\beta$ is a simple root in $J_{s}$ there exists $\tilde w\in W_{J_s}$ such that $\tilde w\beta=\delta$ (where $\delta$ depends on whether $\beta$ is a short or long root). 
Using the fact that $\tilde w\mu$ is a weight of  $L_{J}(w\cdot 0)$, we have 
\begin{equation}\label{dominant2} 
\langle \epsilon_J+\mu+\ell\gamma, \beta^{\vee} \rangle = \langle \epsilon_J,\beta^{\vee} \rangle  + \langle \tilde w\mu, \tilde w\beta^{\vee} \rangle + \langle \ell\gamma, \beta^{\vee} \rangle \geq 
(\ell-1) - (h-1) +0 \geq 0.
\end{equation} 
Consequently, $\epsilon_J+\mu+l\gamma \in X^{+}$ follows by (\ref{dominant1}) and (\ref{dominant2}). 
\end{proof} 

\subsection{} We can now give a positive answer to the 1986 question posed by Friedlander and Parshall (cf. \cite[(3.2)]{FP}) in the case of quantum groups. 

\begin{theorem}\label{thm:quantumrealizereg} Let ${\ell} >h$ and ${\mathcal O}:={\mathcal O}_{J}$ be a 
Richardson orbit in ${\mathcal N}$ with 
moment map $\Gamma:G\times_{P_J} {\mathfrak u}_{J}\rightarrow G\cdot {\mathfrak u}_{J}=\overline{\mathcal O}$. Then there exists 
a finite-dimensional $U_{\zeta}({\mathfrak g})$-algebra $A_{\mathcal O}$ such that as rational $G$-algebras:
\begin{itemize} 
\item[(a)] $\operatorname{H}^{2\bullet+1}(u_{\zeta}({\mathfrak g}),A_{\mathcal O})=0$; 
\item[(b)] $\operatorname{H}^{2\bullet}(u_{\zeta}({\mathfrak g}),A_{\mathcal O})\cong{\bfk}[{\mathcal O}]
\cong {\bfk}[G\times_{P_{J}}{\mathfrak u}_{J}]\cong 
\operatorname{ind}_{P_{J}}^{G} S^{\bullet}({\mathfrak u}_{J}^{*})$.  
\end{itemize} 
Furthermore, if $\Gamma$ is a resolution of singularities and $\overline{\mathcal O}$ is normal then 
$$\operatorname{H}^{2\bullet}(u_{\zeta}({\mathfrak g}),A_{\mathcal O})\cong{\bfk}[\overline{\mathcal O}].$$ 
\end{theorem} 

\begin{proof} We apply Theorem~\ref{thm:quantumnonshiftrealization} 
in the case when $w=\text{id}$ and $\gamma=0$. Then 
if we set 
$A_{\mathcal O}:=Q_{w}=
H^{0}(\epsilon_{J})^{*}\otimes H^{0}(\epsilon_{J})
\cong \text{End}_{\bfk}(H^{0}(\epsilon_J))$ 
which is a 
$U_{\zeta}({\mathfrak g})$-algebra. 
The statements (a) and (b) now follow by
Theorem~\ref{thm:quantumnonshiftrealization} by 
noting that in this specific setting the 
isomorphisms are now as rational $G$-algebras. 

For the last statement of the theorem, the conditions that $\Gamma$ is a resolution of singularities and $\overline{\mathcal O}$ is  
normal imply that ${\bfk}[\overline{\mathcal O}]\cong {\bfk}[{\mathcal O}]$ (cf. \cite[\S 3.5]{BNPP}). 
\end{proof} 

We remark that we have assumed that $\bfk$ is of characteristic 0 and not necessarily algebraically closed. Since the group $G$ is assumed to be defined and split over $\bfk$, the Richardson orbit is defined over any fields and thus the statement of the theorem makes sense. Friedlander and Parshall's question is for algebraic groups in positive characteristic. However,  the question of which closures of nilpotent orbits are normal is still open in types $E_{7}$ and $E_{8}$. For type $A_{n}$ 
all nilpotent orbit closures are normal and for types $B_{n}$, $C_{n}$, $D_{n}$, $E_{6}$, $F_{2}$, and $G_{2}$ this question 
has been resolved (cf. \cite{KP}, \cite{So1, So2}, \cite{Br1}). 

Let $\Phi_{\lambda}=\{\alpha\in \Phi:\ \langle \lambda+\rho,\alpha^{\vee}\rangle\in \ell{\mathbb Z}\}$. 
Under the condition that $(\ell,p)=1$ for any bad prime $p$ of $\Phi$, one can find $w\in W$ such 
that $w(\Phi_{\lambda})=\Phi_{J}$. Moreover, one can apply \cite[Theorem 1.3.5]{BNPP} to see that 
the support variety of $A_{\mathcal O}$ is the closure of the Richardson orbit associated to $J$ 
(i.e.,  ${\mathcal V}_{\mathfrak g}(A_{\mathcal O})=G\cdot {\mathfrak u}_{J}$). In general 
the rate of growth of $\dim\operatorname{H}^{p}(u_{\zeta}({\mathfrak g}),A_{\mathcal O})$ will equal the 
$\dim G\cdot {\mathfrak u}_{J}$. A priori, $\operatorname{H}^{\bullet}(u_{\zeta}({\mathfrak g}),
A_{\mathcal O})$ could be a non-commutative algebra, but Theorem~\ref{thm:quantumrealizereg} demonstrates that the Yoneda product in this case is commutative.  

\subsection{} When $\Phi=A_{n-1}$, the $G=GL_{n}$-orbits are labelled by partitions of $n$. 
If $\lambda$ is a partition of $n$, let $x_{\lambda}$ be the nilpotent matrix in 
$\mathfrak g=\mathfrak{gl}_{n}$ 
with Jordan blocks with sizes matching up with the parts of $\lambda$. Set ${\mathcal O}_{\lambda}=
G\cdot x_{\lambda}$. Given $J\subseteq \Delta$, one can associate a partition $\sigma(J)$ 
to the subroot system generated by $J$ (cf. \cite[Section 4.5]{NPV}). According to \cite{Kr}, 
$$\overline{{\mathcal O}_{\sigma(J)^{t}}}=G\cdot {\mathfrak u}_{J}.$$ 
Here $(-)^{t}$ denotes the transposed partition. 

The centralizers of nilpotent elements under $G$ are connected so $\Gamma$ is a desingularization. 
Moreover, from \cite{KP} all $G$-orbits have normal orbit closures. 
Consequently, Theorem~\ref{thm:quantumrealizereg} can be stated in terms of partitions indexing the various nilpotent orbits. 

\begin{corollary} 
\label{cor:4.4.1} Let ${\mathfrak g}=\mathfrak{gl}_{n}$, ${\ell}>n$ and $\lambda$ be a 
partition of $n$. Let $J$ be a subset of simple roots corresponding to the partition 
$\lambda^{t}$. Then there exists a finite dimensional $U_{\zeta}({\mathfrak g})$-algebra $A_{\mathcal O_{\lambda}}$ 
such that 
\begin{itemize} 
\item[(a)] $\operatorname{H}^{2\bullet+1}(u_{\zeta}({\mathfrak g}),A_{{\mathcal O}_{\lambda}})=0$; 
\item[(b)] $\operatorname{H}^{2\bullet}(u_{\zeta}({\mathfrak g}),A_{{\mathcal O}_{\lambda}})
\cong{\bfk}[G\cdot {\mathfrak u}_{J}]\cong{\bfk}[\overline{\mathcal O}]$.  
\end{itemize} 
\end{corollary} 

\section{Shifted realization of the coordinate ring of the twisted cotangent bundle} 

\subsection{} In this section we will provide a computation involving the cohomology of certain tilting modules for 
$U_{\zeta}({\mathfrak l}_{J})$ and demonstrate that this yields another realization of the coordinate ring of the 
twisted cotangent bundle. This result is 
quite striking in the sense that the 
computation of the characters of the tilting 
modules is only known conjecturally or via 
$p$-Kazhdan-Lusztig polynomials 
(cf. \cite[E.10 (3) Conjecture]{Jan1} 
and \cite[Theorem 7.6]{AMRW}), 
yet we will show that without 
assuming knowledge about the  
character formula one 
can still compute non-trivial cohomology. 

For $\mu\in X_{J}^{+}$, let $T_{J}(\mu)$ be the unique indecomposable tilting $U_{\zeta}({\mathfrak l}_{J})$-module with highest weight $\mu$. recall $\epsilon_{J}=\sum_{\alpha\in J}({\ell}-1)\omega_{\alpha}$. 
Let $\mu\in X_{J,\text{res}}$ be the set of
$\ell$-restricted weights of $X$. Using the
work of Pillen \cite[Section 2, Corollary A]{P}, 
one can show that for $\mu\in X_{J,\text{res}}$ 
the tilting module 
$T_{J}(\epsilon_{J}-w_{0,J}\epsilon_{J}+w_{0,J}\mu)$ 
can be realized as a multiplicity one
$U_{\zeta}({\mathfrak l}_{J})$-summand in
$L_{J}(\epsilon_{J})\otimes L_{J}(-w_{0,J}\epsilon_{J}+w_{0,J}\mu).$  
Set $\mu_{J}=-w_{0,J}\epsilon_{J}+w_{0,J}
\mu.$ We also review the important fact 
(cf. \cite[ II 10.15 Lemma]{Jan1}) 
that 
\begin{equation} 
\label{importantHomidentity}
\text{Hom}_{u_{\zeta}({\mathfrak l}_{J})}(L(\mu),L_{J}(\epsilon_{J})\otimes L_{J}(\mu_{J}))\cong 
\text{Hom}_{U_{\zeta}({\mathfrak l}_{J})}(L(\mu),L_{J}(\epsilon_{J})\otimes L_{J}(\mu_{J})).
\end{equation} 

\begin{theorem}\label{thm:quantumshiftrealization} Let ${\ell}>h$ and $J\subseteq \Delta$. Set $\sigma_{w,J}=
\epsilon_{J}-w_{0,J}\epsilon_{J}+w_{0,J}(w\cdot 0)$ for  $w\in {}^{J}W$.
Moreover, let $\gamma\in X^{+}$ such that all weights of $\widehat{T}:=T_{J}(\sigma_{w,J})\otimes \bfk_{\ell\gamma}$ are in $X^{+}$.
Then there exists an isomorphism of rational $G$-modules:
\begin{equation}
\operatorname{H}^{n}(u_{\zeta}({\mathfrak g}),\operatorname{ind}_{U_{\zeta}({\mathfrak p}_{J})}^{U({\mathfrak g})} \widehat{T}) \cong
\begin{cases} \operatorname{ind}_{P_{J}}^{G} S^{\frac{n-l(w)}{2}}({\mathfrak u}^{*}_{J}) \otimes \bfk_{\gamma} & n-l(w) \equiv 0 \ (\operatorname{mod }2) ,  \\
0 & \text{otherwise}. 
\end{cases} 
\end{equation} 

\end{theorem} 

\begin{proof} Let $M={\bfk}$ and $Q_w=T_{J}(\epsilon_{J}-w_{0,J}\epsilon_{J}+w_{0,J}(w\cdot 0))$. Recall that if ${\ell}>h$ then $w\cdot 0\in X_{J,\text{res}}$ by 
\cite[Prop. 3.6.1]{UGAVIGRE1}. Therefore, 
$Q_w$ is a $U_{\zeta}({\mathfrak l}_{J})$-summand 
in $L_{J}(\epsilon_{J})\otimes L_{J}(-w_{0,J}\epsilon_{J}+w_{0,J}(w\cdot 0))$, 
and thus $Q_w$ is injective over
$u_{\zeta}({\mathfrak l}_{J})$.  
The statement of the theorem will follow if
we can verify conditions (i) and (ii) of
Theorem~\ref{thm:quantumrealization}. 

Condition (ii) holds because we have 
assumed that all weights of
$T_{J}(\epsilon_{J}-w_{0,J}\epsilon_{J}+w_{0,J}(w\cdot 0))\otimes \bfk_{\ell\gamma}$ 
are in $X^{+}$. Thus, we only need to verify condition (i) which is 
equivalent to showing that 
$$
\operatorname{Hom}_{u_{\zeta}({\mathfrak l}_{J})}(L_{J}(\nu)^{[1]},\operatorname{H}^{n}({\mathcal U}_{\zeta}({\mathfrak u}_{J}),{\bfk})  \otimes  Q_w)\cong 
\begin{cases} {\bfk} & n=l(w)\ \text{and}\ \nu=0   \\
0 & \text{otherwise}.  
\end{cases} 
$$ 
If this Hom-space is non-zero, by Lemma~\ref{coh-direct-summand} there exists $y\in {}^J W$ with $l(y)=n$ such that 
$\operatorname{Hom}_{u_{\zeta}({\mathfrak l}_{J})}(L_{J}(y\cdot 0)\otimes L_{J}(\nu)^{[1]},Q_w)\neq 0$. 
The simple $U_{\zeta}({\mathfrak l}_{J})$-module $L_{J}(w\cdot 0)$ appears in the socle (and head) of $Q_w$. 
This implies that $y\cdot 0+{\ell}\nu$ is linked to $w\cdot 0$ under the action of the affine Weyl group associated to $\Phi_{J}$. 
Therefore, 
$$y\cdot 0+{\ell}\nu=x\cdot (w\cdot 0)+{\ell} \sigma.$$ 
where $\sigma\in {\mathbb Z}\Phi_{J}$ and $x\in W_{J}$. Rewriting this equation we have 
$$y\cdot 0=(xw)\cdot 0+{\ell}(\sigma-\nu).$$
We can use the arguments given in the proof of Theorem~\ref{thm:quantumnonshiftrealization} to conclude that $\sigma=\nu$ and $y=xw$ (or see \cite[Lemmas 2.1.1 and 2.1.2]{DNP}). But 
$y$ and $w$ are in $^JW$ so $y=w$, and $x=\text{id}$. 

Now by (\ref{importantHomidentity}) and the fact that $L_{J}(w\cdot 0)$ appears in the $U_{\zeta}({\mathfrak l}_{J})$-socle and the 
$u_{\zeta}({\mathfrak l}_{J})$-socle of $L_{J}(\epsilon_{J})\otimes L_{J}(-w_{0,J}\epsilon_{J}+w_{0,J}(w\cdot 0))$ exactly once we 
can conclude that $L_{J}(w\cdot 0+l\nu)$ is a composition factor in $Q_w$ only if $\nu=0$. This verifies condition (i). 
\end{proof}

\subsection{} We will first indicate how Theorem~\ref{thm:quantumshiftrealization}  can be viewed as a quantum analog of the  computation of the 
cohomology of the first Frobenius kernel with coefficients in an induced module. This result was first proved by Andersen and Jantzen in most cases 
and by Kumar, Lauritzen, and Thomsen in general (cf. \cite[Corollary 3.7(b)]{AJ} \cite[Theorem 8]{KLT}). 

\begin{corollary} If ${\ell}>h$ and $w\in W$ such that $w\cdot 0+{\ell}\nu\in X^{+}$.  Then there exists an isomorphism of rational $G$-modules 
\begin{equation} 
\operatorname{H}^{n}(u_{\zeta}({\mathfrak g}),\operatorname{ind}_{U_{\zeta}({\mathfrak b})}^{U_{\zeta}({\mathfrak g})} [\bfk_{w\cdot 0+\ell\gamma}])\cong
\begin{cases} \operatorname{ind}_{B}^{G} S^{\frac{n-l(w)}{2}}({\mathfrak u}^{*})\otimes \bfk_{\gamma} & \text{ if } n-l(w)\equiv 0 \ (\on{mod }2), \\
0 & \text{otherwise}. 
\end{cases} 
\end{equation} 
\end{corollary} 

\begin{proof} The statement of the corollary is a special case of Theorem~\ref{thm:quantumshiftrealization} in the case when $J=\varnothing$. Then ${\mathfrak p}_{J}={\mathfrak b}$, 
$\epsilon_{J}=0$, $w_{0,J}=\text{id}$, and $w\in W$. Then $T_{J}(\epsilon_{J}-w_{0,J}\epsilon_{J}+w_{0,J}(w\cdot 0))=w\cdot 0$. 
\end{proof}

\section{Connections with cohomology for quantum groups}

\subsection{} In this section we will indicate the strong interplay between the structure of geometric objects in the setting of complex Lie theory and homological information 
for modules over the quantum group. The first result of this section connects the multiplicities of simple $G$-modules in the 
global sections of the twisted cotangent bundle with the cohomology for modules over the large quantum group. 

\begin{theorem}\label{thm:cohoquantumGcompfactors} Let $J\subseteq \Delta$, $M$ be a $U_{\zeta}({\mathfrak g})$-module with $Z_{J}.M=0$ and $Q$ be a $U_{\zeta}({\mathfrak p}_{J})$-module such that 
$Q$ is an injective $u_{\zeta}({\mathfrak l}_{J})$-module which is trivial as $U_{\zeta}({\mathfrak u}_{J})$-module. Assume that the following 
two conditions hold: 
\begin{itemize} 
\item[(i)] As a rational $B$-module,   
\begin{equation}
\operatorname{Hom}_{u_{\zeta}({\mathfrak l}_{J})}({\bfk},\operatorname{H}^{n}({\mathcal U}_{\zeta}({\mathfrak u}_{J}), M)\otimes Q)\cong
\begin{cases} k & \text{ if } n=t, \\
0 & \text{otherwise.} 
\end{cases} 
\end{equation} 
\item[(ii)] For $\gamma\in X_{P_{J}}\cap X^{+}$, $R^{i}\operatorname{ind}_{U_{\zeta}({\mathfrak p}_{J})}^{U_{\zeta}({\mathfrak g})} (Q \otimes \bfk_{\ell\gamma}) =0$ for $i>0$. 
Then for $n\geq 0$ and $ \sigma\in X^+$, the composition factor multiplicities as $G$-modules are given by
\end{itemize} 
$$[\operatorname{ind}_{P_{J}}^{G}( S^{\frac{n-t}{2}}({\mathfrak u}_{J}^{*})\otimes \bfk_{\gamma}) :L(\sigma)]=
\begin{cases} \dim \on{Ext}^{n-t}_{U_{\zeta}({\mathfrak g})}
(L(\sigma)^{[1]}, \widehat{M} )   & \text{ if } n-t \equiv 0 \ (\operatorname{mod}\ 2),       \\
0   & \text{otherwise} 
\end{cases} 
$$ 
where $\widehat{M}=M\otimes \on{ind}_{U_{\zeta}({\mathfrak p}_{J})}^{U_{\zeta}({\mathfrak g})}(Q \otimes \bfk_{\ell\gamma})$.
\end{theorem}

\begin{proof} One can apply the Lyndon-Hochschild-Serre spectral sequence for $u_{\zeta}({\mathfrak g})$, a normal Hopf subalgebra in 
$U_{\zeta}({\mathfrak g})$, to obtain: 
$$E_{2}^{i,j}
=\text{Ext}^{i}_{G}(L(\sigma),\text{Ext}^{j}_{u_{\zeta}({\mathfrak g})}({\bfk},\widehat{M}))
\Rightarrow \text{Ext}^{i+j}_{U_{\zeta}({\mathfrak g})}
(L(\sigma)^{[1]},M\otimes \operatorname{ind}_{U_{\zeta}({\mathfrak p}_{J})}^{U_{\zeta}({\mathfrak g})}(Q \otimes \bfk_{\ell\gamma})).$$ 
The finite-dimensional $G$-modules are completely reducible so there are no higher extensions between $G$-modules. This implies that this spectral sequence collapses and yields: 
\begin{eqnarray*}
[\text{H}^{\bullet}(u_{\zeta}({\mathfrak g}),\widehat{M}):L(\sigma)]
&=&
\dim\on{Hom}_{G}(L(\sigma),\on{H}^{\bullet}(u_{\zeta}({\mathfrak g}),M \otimes  \on{ind}_{U_{\zeta}({\mathfrak p}_{J})}^{U_{\zeta}({\mathfrak g})}(Q \otimes \bfk_{\ell\gamma})))\\
&=& \dim\on{Ext}^{\bullet}_{U_{\zeta}({\mathfrak g})}(L(\sigma)^{[1]},  M\otimes \on{ind}_{U_{\zeta}({\mathfrak p}_{J})}^{U_{\zeta}({\mathfrak g})}(Q \otimes \bfk_{\ell\gamma})).
\end{eqnarray*} 
The result now follows from Theorem~\ref{thm:quantumrealization}. 
\end{proof} 

We note that the multiplicities $[\operatorname{ind}_{P_{J}}^{G} S^{\bullet}({\mathfrak u}_{J}^{*})\otimes \bfk_{\gamma}:L(\sigma)]$ are 
given in terms of classical formulas involving values of Kostant's partition function (cf. \cite[8.18]{Jan3}). 
These values can also be interpreted as certain coefficients of Kazhdan-Lusztig polynomials. 
The theorem above can now be used to give an interpretation of the dimensions of the extension groups 
$$\text{Ext}^{\bullet}_{U_{\zeta}({\mathfrak g})}(L(\sigma)^{[1]},  
M\otimes \operatorname{ind}_{U_{\zeta}({\mathfrak p}_{J})}^{U_{\zeta}({\mathfrak g})}(Q\otimes \bfk_{\ell\gamma}))$$
via values on the Kostant's partition function.

\subsection{} Theorem~\ref{thm:cohoquantumGcompfactors} can be combined with Theorem~\ref{thm:quantumnonshiftrealization} to yield the following result.  

\begin{corollary} \label{cor:nonshiftedextrealization} Let ${\ell}>h$, $J\subseteq \Delta$, $\epsilon_{J}=\sum_{\alpha\in J}({\ell}-1)\omega_{\alpha}$, and $M=H^{0}(\epsilon_J)^{*}$. 
For $w\in {}^JW$ and $\gamma\in X_{P_{J}}\cap X^{+}$, set 
$$N_w(\gamma):=H^{0}(\epsilon_{J})^{*}\otimes  \operatorname{ind}_{U_{\zeta}({\mathfrak p}_{J})}^{U_{\zeta}({\mathfrak g})}( H^{0}_{J}(\epsilon_{J})\otimes 
H^{0}_{J}(-w_{0,J}w\cdot 0)\otimes \bfk_{\ell\gamma}),$$
and assume that 
\begin{itemize}
\item[(a)] $\operatorname{Hom}_{u_{\zeta}({\mathfrak l}_{J})}({\bfk},\operatorname{H}^{0}({\mathcal U}_{\zeta}({\mathfrak u}_{J}), M)\otimes Q_w)\cong \bfk$,
\item[(b)] $\epsilon_J-w_{0,J}(w\cdot 0)+\ell\gamma \in X^{+}$.  
\end{itemize} 
Then for $n\geq 0$ 
\begin{equation*}
[\on{ind}_{P_{J}}^{G} (S^{\frac{n}{2}}({\mathfrak u}_{J}^{*})^{[1]}\otimes \bfk_{\gamma}) :L(\sigma)]=
\begin{cases} \dim \on{Ext}^{n}_{U_{\zeta}({\mathfrak g})}
(L(\sigma)^{[1]},  N_w(\gamma))  
&\text{ if } w=\on{id} \text{and}      \\
&\   n \equiv 0 \ (\on{mod }2),   \\
0   & \text{otherwise}.
\end{cases} 
\end{equation*} 
\end{corollary} 

As a special case of the result above one sees that the multiplicities of simple $G$-modules in 
the graded coordinate algebra of (conical) orbit closures are directly related to the cohomology of quantum groups. 

\begin{corollary}\label{cor:cohoquantum} 
Let ${\ell}>h$ and ${\mathcal O}:={\mathcal O}_{J}$ be a Richardson orbit in ${\mathcal N}$ with 
moment map $\Gamma:G\times_{P_J} {\mathfrak u}_{J}\rightarrow G\cdot {\mathfrak u}_{J}$. Set 
$\epsilon_{J}=({\ell}-1)\sum_{\alpha\in J}\omega_{\alpha}$. Then 
\begin{itemize}
\item[(a)] $[{\bfk}[G\times_{P_{J}}{\mathfrak u}_{J}]_\bullet:L(\sigma)]=\dim \operatorname{Ext}^{2\bullet}_{U_{\zeta}({\mathfrak g})}
(H^{0}(\epsilon_{J})\otimes L(\sigma)^{[1]},H^{0}(\epsilon_{J})).$
\item[(b)] Furthermore, if $\Gamma$ is a resolution of singularities and $\overline{\mathcal O}$ is normal then 
$$[{\bfk}[\overline{\mathcal O}]_{\bullet}:L(\sigma)]=\dim \operatorname{Ext}^{2\bullet}_{U_{\zeta}({\mathfrak g})}
(H^{0}(\epsilon_{J})\otimes L(\sigma)^{[1]},H^{0}(\epsilon_{J}))$$
where $k[\overline{\mathcal O}]_{\bullet}$ is the (graded) coordinate algebra of the orbit closure $\overline{\mathcal O}$ which is a conical variety.
\item[(c)] In particular $[{\bfk}[{\mathcal N}]_{\bullet}:L(\sigma)]=\dim \operatorname{Ext}^{2\bullet}_{U_{\zeta}({\mathfrak g})}
(L(\sigma)^{[1]},{\bfk})$.  
\end{itemize} 
\end{corollary}

We can also use Theorem~\ref{thm:cohoquantumGcompfactors} in conjunction with Theorem~\ref{thm:quantumshiftrealization} to realize composition factor multiplicities 
in a shifted version of quantum group cohomology. 

\begin{corollary} \label{cor:shiftedextrealization} Let ${\ell}>h$, $J\subseteq \Delta$, $w\in {}^{J}W$ and $\epsilon_{J}=({\ell}-1)\sum_{\alpha\in J}\omega_{\alpha}$.
Moreover, let $\gamma\in X^{+}$ such that all weights of $T_{J}(\epsilon_{J}-w_{0,J}\epsilon_{J}+w_{0,J}(w\cdot 0))\otimes \bfk_{\ell\gamma}$ are in $X^{+}$. 
 Set 
$$N_w=\operatorname{ind}_{U_{\zeta}({\mathfrak p}_{J})}^{U_\zeta({\mathfrak g})} (T_{J}(\epsilon_{J}-w_{0,J}\epsilon_{J}+w_{0,J}(w\cdot 0))\otimes \bfk_{\ell\gamma}).$$  
Then
\begin{equation*}
[\widehat{I} :L(\sigma)]=
\begin{cases} \dim \operatorname{Ext}^{n}_{U_{\zeta}({\mathfrak g})}
(L(\sigma)^{[1]},  N_w(\gamma))  
&  
\text{ if } n-l(w) \equiv 0 \ (\operatorname{mod }2) ,      \\
0   & \text{otherwise}
\end{cases} 
\end{equation*} 
where $\widehat{I}=\operatorname{ind}_{P_{J}}^{G} (S^{\frac{n-l(w)}{2}}({\mathfrak u}_{J}^{*})\otimes \bfk_{\gamma})$. 
\end{corollary}

\section{Frobenius kernels}
 
\subsection{} We will demonstrate which results we have proved 
for quantum groups will hold for Frobenius kernels (under suitable cohomological vanishing conditions). Let $G$ be a reductive 
algebraic group scheme defined and split over ${\mathbb F}_{p}$ and $k$ be an algebraically 
closed field of characteristic $p>0$. 
Moreover, in this section, for any algebraic group $H$ 
defined over $\mathbb F_p$, 
let $H_{1}$ be the scheme theoretic 
kernel of the Frobenius map $\on{Fr}:H\rightarrow H$. 
If $ J\subseteq \Delta$, we will use
$P_J=L_J\ltimes U_J$ to denote the
parabolic subgroup of $G$ corresponding
to $J$ with a Levi subgroup $ L_J$ 
and the  unipotent radical $U_J$. 
All of these subgroups are defined over
$ \mathbb F_p$ and are $\on{Fr}$-stable.
As before we use $ \mathfrak p_J$, $\mathfrak u_J$, and $\mathfrak l_J$ to denote the corresponding Lie algebras over $k$. 
Following \cite[\S 7.2]{BNPP} we state 
the following assumptions on 
$J\subseteq \Delta$. 
The first assumption entails the 
Grauert-Riemenschneider 
vanishing result: 
\vskip .25cm 
\noindent
(A1) $R^{i}\text{ind}_{P_{J}}^{G} (S^{\bullet}({\mathfrak u}_{J}^{*})\otimes k_{\gamma})=0$ for $i>0$ and $\gamma\in X_{P_{J}}\cap X^{+}$, and 
$\text{ind}_{P_{J}}^{G} (S^{\bullet}({\mathfrak u}_{J}^{*})\otimes k_{\gamma})$ has a good filtration. 
\vskip .25cm 

\noindent
The second assumption on $J$ is a condition on the normality of $G\cdot {\mathfrak u}_{J}$. 
\vskip .25cm 
\noindent
(A2) The variety $G\cdot {\mathfrak u}_{J}$ is normal. 
\vskip .25cm 
\noindent

We note that (A1) have been shown when $p$ is a good prime and $\gamma$ satisfies the additional condition that 
$\langle \gamma, \alpha^{\vee} \rangle >0$ for all $\alpha\in \Delta-J$ \cite[Theorem 5]{KLT}. The condition $(A2)$, for example, has been verified when $\Phi=A_{n}$ 
(see \cite{Don}). 

\subsection{} Next we indicate how one can prove the analogue of 
Theorem~\ref{thm:quantumrealization}. 
For Frobenius kernels the proof is simpler because one has 
natural actions of the Levi subgroup 
on the exterior algebra and the ordinary
Lie algebra cohomology. 

We recall that $\on{H}^*(\mathfrak h, -)$ is the Lie algebra cohomology while $\on{H}^*(H_1, -)=\on{Ext}^{*}_{H_1}(k, -)$ for any algebraic group $H$ defined over $\mathbb F_p$.

\begin{theorem}\label{thm:Frobrealization} 
Let $p\geq 3$, $J\subseteq \Delta$, $M$ be a $G$-module, and 
$Q$ be a $P_{J}$-module such that 
$Q$ is an injective $(L_{J})_{1}$-module and trivial as a $U_{J}$-module. 
Assume that (A1) and the following 
two conditions hold: 
\begin{itemize} 
\item[(i)] As a rational $P_{J}/(P_{J})_{1}$-module, and for some $t\geq 0$, 
\begin{equation}
\operatorname{Hom}_{(L_{J})_{1}}(k,\operatorname{H}^{n}({\mathfrak u}_{J}, M)\otimes Q)\cong
\begin{cases} k & \text{ if } n=t, \\
0 & \text{otherwise.} 
\end{cases} 
\end{equation} 
\item[(ii)] For $\gamma\in X_{P_{J}}\cap X^{+}$,
$R^{i}\operatorname{ind}_{P_{J}}^{G} (Q \otimes k_{p\gamma}) =0$ for $i>0$. 
\end{itemize} 
Then there exists an isomorphism of rational $G$-modules:
\begin{equation}
\operatorname{H}^{n}(G_{1},M\otimes \operatorname{ind}_{P_{J}}^{G} (Q\otimes p\gamma))\cong
\begin{cases} \operatorname{ind}_{P_{J}}^{G} (S^{\frac{n-t}{2}}({\mathfrak u}^{*}_{J}) \otimes k_{\gamma})^{[1]} & n-t \equiv 0 \ (\operatorname{mod }2) \\
0 & \text{otherwise} 
\end{cases} 
\end{equation} 
\end{theorem}

\begin{proof} In order to prove the theorem, we apply the LHS spectral sequence 
to $(U_{J})_{1}$ embedded in $(P_{J})_{1}$ as a normal subgroup scheme with quotient $(L_{J})_{1}$ using the fact that $Q$ is $U_J$-trivial: 
\begin{equation*} 
E_{2}^{i,j}=\operatorname{H}^{i}((L_{J})_{1},\operatorname{H}^{j}((U_{J})_{1},M) \otimes Q )\Rightarrow 
\operatorname{H}^{i+j}((P_{J})_{1},M\otimes Q).
\end{equation*} 
Since $Q$ is projective over $(L_{J})_{1}$, this spectral sequence collapses and 
yields: 
\begin{equation*} 
\text{H}^{n}((P_{J})_{1},M\otimes Q)\cong \text{Hom}_{(L_{J})_{1}}(k,\operatorname{H}^{n}((U_{J})_{1},M)\otimes Q ). 
\end{equation*}   
For $p\geq 3$, there exists a first quadrant spectral sequence \cite[(1.3) Proposition]{FP2}:
\begin{equation*}
{E}_{2}^{2i,j}=S^{i}({\mathfrak u}_{J}^{*})^{[1]}\otimes
\text{H}^{j}({\mathfrak u}_{J},M\otimes Q)
\Rightarrow \text{H}^{2i+j}((U_{J})_{1},M\otimes Q).
\end{equation*}
Since the functor $\text{Hom}_{(L_{J})_{1}}(k, -\otimes Q)$ is exact, we can
compose it with the spectral sequence above and use the fact that $S^*(\mathfrak u_J^*)^{(i)}$ is a trivial $(P_J)_1$-module and $Q$ is a trivial module when restricted to $U_J$ to get another spectral sequence:
\begin{equation}
E_{2}^{2i,j}=S^{i}({\mathfrak u}_{J}^{*})^{[1]}\otimes
\text{Hom}_{(L_{J})_{1}}(k,\text{H}^{j}({\mathfrak u}_{J},M)\otimes Q)\Rightarrow
\text{H}^{2i+j}((P_{J})_{1},M\otimes Q).
\end{equation}
Condition (i) implies that the $E_{2}$ only lives one (horizontal) line, thus the spectral sequence collapses and yields 
\begin{equation*} 
\operatorname{Hom}_{(L_{J})_{1}}(k,\operatorname{H}^{n}((U_{J})_{1},M) \otimes Q)) 
\cong \begin{cases} S^{\frac{n-t}{2}}({\mathfrak u}^{*}_{J})^{[1]}  & \text{ if } n\equiv t \ (\text{mod }2),  \\
0 & \text{otherwise}. 
\end{cases} 
\end{equation*} 
and  
\begin{equation*}
\operatorname{H}^{n}((P_{J})_{1},M\otimes Q)\cong
\begin{cases} S^{\frac{n-t}{2}}({\mathfrak u}^{*}_{J})^{[1]}  & \text{ if } n-t\equiv 0 \ (\operatorname{mod }2), \\
0 & \text{ otherwise}.
\end{cases} 
\end{equation*} 

Now one can apply the same techniques given in the proof of Theorem~\ref{thm:quantumrealization} and (A1) to 
yield the desired results. 

\end{proof}

\subsection{} By using Theorem~\ref{thm:Frobrealization} one can prove Frobenius kernel analogs (using the same weight arguments as given in the quantum case) of 
Theorems ~\ref{thm:quantumnonshiftrealization}, ~\ref{thm:quantumrealizereg}, and ~\ref{thm:quantumshiftrealization}. 

\begin{theorem}\label{thm:Frobnonshiftrealization} 
Let $p>2h-1$, $J\subseteq \Delta$, $\epsilon_{J}
=\sum_{\alpha\in J}(p-1)\omega_{\alpha}$, $M=H^0(\epsilon_J)^*$.  
 For $w\in {}^JW$ and $\gamma\in X_{P_{J}}\cap X^{+}$ 
 set 
$$N_{w}(\gamma):=H^{0}(\epsilon_{J})^{*}\otimes  \operatorname{ind}_{P_{J}}^{G}( H^{0}_{J}(\epsilon_{J})\otimes 
H^{0}_{J}(-w_{0,J}w\cdot 0)\otimes k_{p\gamma}).$$
Assume (A1) holds and  
\begin{itemize}
\item[(a)] $\operatorname{Hom}_{(L_{J})_{1}}(k,\operatorname{H}^{0}((U_{J})_{1},M) \otimes Q))\cong k$,
\item[(b)] $\epsilon_J-w_{0,J}(w\cdot 0)+p\gamma \in X^{+}$.  
\end{itemize} 
Then there exists an isomorphism of rational $G$-modules:
\begin{equation}
\on{H}^{n}(G_{1},N_{w}(\gamma))\cong
\begin{cases} \on{ind}_{P_{J}}^{G} (S^{\frac{n}{2}}({\mathfrak u}^{*}_{J}) \otimes k_{\gamma})^{[1]} & \text{ if } w=\on{id } \text{ and } \ n \equiv 0 \ (\on{mod }2) ,  \\
0 & \text{ otherwise}. 
\end{cases} 
\end{equation} 
In the case when $w=\on{id}$ for $N_{w}(\gamma)$, the condition on $p$ can be replaced by $p>h$. 
\end{theorem}

This theorem specializes to the case when $\gamma=0$ and $w=\text{id}$ to answer the question posed by Friedlander and 
Parshall for $G_{1}$ assuming that $(A1)$ and $(A2)$ hold. 

\begin{theorem}\label{thm:Frobrealization2} Let $p>h$, $J\subseteq \Delta$ and assume that 
(A1) holds. Let ${\mathcal O}:={\mathcal O}_{J}$ be a 
Richardson orbit in ${\mathcal N}$ with 
moment map $\Gamma:G\times_{P} {\mathfrak u}_{J}\rightarrow G\cdot {\mathfrak u}_{J}$. Then there exists 
a finite-dimensional $G$-algebra $A_{\mathcal O}$ such that 
\begin{itemize} 
\item[(a)] $\operatorname{H}^{2\bullet+1}(G_{1},A_{\mathcal O})=0$; 
\item[(b)] $\operatorname{H}^{2\bullet}(G_{1},A_{\mathcal O})\cong{k}[{\mathcal O}]^{[1]}
\cong {k}[G\times_{P_{J}}{\mathfrak u}_{J}]^{[1]}\cong 
\operatorname{ind}_{P_{J}}^{G} S^{\bullet}({\mathfrak u}_{J}^{*})^{[1]}$.  
\end{itemize} 
Furthermore, if $\Gamma$ is a resolution of singularities and $(A2)$ holds then 
$$\operatorname{H}^{2\bullet}(u_{\zeta}({\mathfrak g}),A_{\mathcal O})\cong{k}[\overline{\mathcal O}]^{[1]}.$$ 
\end{theorem} 
In the following theorem, we will use the notation $ T_J(\lambda)$ to denote the partial tilting $L_J$-module of hightest weight  $ \lambda \in X_J^+$. 
\begin{theorem}\label{thm:Frobshiftrealization} 
Let $p>h$, $J\subseteq \Delta$, $w\in {}^{J}W$ and
$\epsilon_{J}=(p-1)\sum_{\alpha\in J}\omega_{\alpha}$. 
Assume that (A1) holds. 
Moreover, let $\gamma\in X_{P_J}\cap X^{+}$ such that all weights of $T_{J}(\epsilon_{J}-w_{0,J}\epsilon_{J}+w_{0,J}(w\cdot 0))\otimes k_{p\gamma}$ are in $X^{+}$
Then there exists an isomorphism of rational $G$-modules:
\begin{equation}
\on{H}^{n}(G_{1},\on{ind}_{P_{J}}^{G} (\widehat{T})\otimes k_{p\gamma}))\cong
\begin{cases} \on{ind}_{P_{J}}^{G}( S^{\frac{n-l(w)}{2}}({\mathfrak u}^{*}_{J}) \otimes k_{\gamma})^{[1]} & n-l(w) \equiv 0 \ (\on{mod }2),   \\
0 & \text{ otherwise}
\end{cases} 
\end{equation} 
where $\widehat{T}=T_{J}(\epsilon_{J}-w_{0,J}\epsilon_{J}+w_{0,J}(w\cdot 0))$. 

\end{theorem}

\subsection{} Corollaries ~\ref{cor:nonshiftedextrealization}, ~\ref{cor:cohoquantum}, and ~\ref{cor:shiftedextrealization} have analogs in the case of Frobenius kernels 
via the use of good filtrations. We recall that a $G$-module $E$ is said to have a good filtration if $E$ has a filtration with sections isomorphism to modules of the form $ H^0(\sigma)$ with $ \sigma \in X^+$. The number of times $H^0(\sigma)$ appears as the sections is independent of the choice of the good filtration and denoted by $[E: H^0(\sigma)]$. This notation (not to be confused with composition factor multiplicities) will only be used in this section. 

\begin{corollary} \label{cor:Frobcompfactor1} 
Let $p>h$, $J\subseteq \Delta$, 
$\epsilon_{J}=\sum_{\alpha\in J}(p-1)\omega_{\alpha}$, and $M=H^0(\epsilon_J)^*$. For 
 $w\in {}^JW$ and $\gamma\in X_{P_{J}}\cap X^{+}$, set 
$$N_w(\gamma):=H^{0}(\epsilon_{J})^{*}\otimes  \operatorname{ind}_{P_{J}}^{G}( H^{0}_{J}(\epsilon_{J})\otimes 
H^{0}_{J}(-w_{0,J}w\cdot 0)\otimes k_{p\gamma}).$$
Assume that (A1) holds and  
\begin{itemize}
\item[(a)] $\operatorname{Hom}_{(L_{J})_{1}}(k,\operatorname{H}^{0}((U_{J})_{1},M) \otimes Q))\cong k$,
\item[(b)] $\epsilon_J-w_{0,J}(w\cdot 0)+p\gamma \in X^{+}$.  
\end{itemize} 
Then 
\begin{equation*}
[\widehat{I} :H^{0}(\sigma)]=
\begin{cases} \dim \on{Ext}^{n}_{G}
(V(\sigma)^{[1]},  N_w(\gamma))  
& \text{ if } w=\on{id} 
\text { and }  \ n \equiv 0 \ (\on{mod }2),       \\
0   & \text{ otherwise}
\end{cases} 
\end{equation*} 
where $\widehat{I}=\on{ind}_{P_{J}}^{G} (S^{\frac{n}{2}}({\mathfrak u}_{J}^{*})^{[1]}\otimes k_{\gamma})$. 
\end{corollary} 

\begin{corollary}\label{cor:Frobcompfactor2} Assume that (A1)  holds.
Let $p>h$ and ${\mathcal O}:={\mathcal O}_{J}$ be a Richardson orbit in ${\mathcal N}$ with 
moment map $\Gamma:G\times_{P_J} {\mathfrak u}_{J}\rightarrow G\cdot {\mathfrak u}_{J}$. Set 
$\epsilon_{J}=(p-1)\sum_{\alpha\in J}\omega_{\alpha}$. Then 
\begin{itemize}
\item[(a)] $[{k}[G\times_{P_{J}}{\mathfrak u}_{J}]:H^{0}(\sigma)]=\dim \operatorname{Ext}^{2\bullet}_{G}
(H^{0}(\epsilon_{J})\otimes V(\sigma)^{[1]},H^{0}(\epsilon_{J})).$
\item[(b)] Furthermore, if $\Gamma$ is a resolution of singularities and $(A2)$ holds then 
$$[{k}[\overline{\mathcal O}]_{\bullet}:H^{0}(\sigma)]=\dim \operatorname{Ext}^{2\bullet}_{G}
(H^{0}(\epsilon_{J})\otimes V(\sigma)^{[1]},H^{0}(\epsilon_{J})).$$
In particular when $J=\varnothing$,  (A1) and (A2) are  satisfied, and the moment map is a resolution of singularities, then
$$[{k}[{\mathcal N}]_{\bullet}:H^{0}(\sigma)]=\dim \operatorname{Ext}^{2\bullet}_{G}(V(\sigma)^{[1]},{\mathbb C}).$$  
\end{itemize} 
\end{corollary}

\begin{corollary} \label{cor:Frobcompfactor3} 
Let $p>h$, $J\subseteq \Delta$, $w\in {}^{J}W$ 
and $\epsilon_{J}=(p-1)\sum_{\alpha\in J}\omega_{\alpha}$.
Assume that (A1) holds.  
Moreover, let $\gamma\in X^{+}$ such that all weights of $T_{J}(\epsilon_{J}-w_{0,J}\epsilon_{J}+w_{0,J}(w\cdot 0))\otimes k_{p\gamma}$ are in $X^{+}$. Set 
$$N_w=\operatorname{ind}_{U_{\zeta}({\mathfrak p}_{J})}^{U({\mathfrak g})} (T_{J}(\epsilon_{J}-w_{0,J}\epsilon_{J}+w_{0,J}(w\cdot 0))\otimes k_{p\gamma}).$$ 
Then 
\begin{equation*}
[\widehat{I} :H^{0}(\sigma)]=
\begin{cases} \dim \operatorname{Ext}^{n}_{G}
(V(\sigma)^{[1]},  N_w)  
&  \text{ if }\ n-l(w) \equiv 0 \ (\operatorname{mod }2),       \\
0   & \text{ otherwise}
\end{cases} 
\end{equation*} 
where $\widehat{I}=\operatorname{ind}_{P_{J}}^{G} (S^{\frac{n-l(w)}{2}}({\mathfrak u}_{J}^{*})\otimes k_{\gamma})$. 
\end{corollary}

\section{Appendix: Hopf Algebra Actions on Cohomology} \label{Hopf-algebra-action}

\subsection{}\label{sec:appendix} Let $H$ be an arbitrary Hopf algebra over a commutative ring $\bfk$. As in  \cite[4.1]{Lin1} and \cite{LN}, the left adjoint action of   $H$  on itself  is defined by  
\[ h\cdot x=\sum_{(h)} h_{(1)}xS(h_{(2)}) \]
for all $ h\in H$ and $x\in R$ with $\Delta(h)=\sum_{(h)}h_{(1)}\otimes h_{(2)}$ (using the Sweedler notation). Here $S$ is the antipode. 
It is easily checked that 
\[ h\cdot (x_1x_2)=\sum_{(h)}(h_{(1)}\cdot x_1)(h_{(2)}\cdot x_2)\]
 making $H$ into an $H$-module algebra. Any  (unital) $\bfk$-subalgebra $D\subseteq H$ stable 
	under this action is a normal subalgebra of $H$ regarded as an augmented (supplemented) algebra 
	in sense of \cite{CE} with the augmentation $\epsilon: D\to R$ being the restriction of the counit of $H$.  One can similarly define the right adjoint action making $H$ as a right $H$-module.

\subsection{} \label{sec:2.5} In \cite{LN}, a Hopf algebra action on cohomology  of a 
Hopf algebra is defined. More general Hopf algebra actions on cohomology of 
augmented algebras were given in \cite{BNPP}. 
We outline the setup here for later use in this section (and in this paper).  For simplicity, all algebras in the rest of this section are over an arbitrary field $\bfk$.  Let $H$ be a Hopf algebra and $A$ be an augmented algebra with a {\em right} $H$-module structure, i.e., 
$A$ is a right $H$-module such that both the
multiplication $ A\otimes A\rightarrow A$, 
the augmentation 
$ \epsilon: A\rightarrow \bfk$, and the unit
$u: \bfk\to A$  are $H$-module homomorphisms.
In such case, we call $A$ a right $H$-module
augmented algebra. This is equivalent that 
$A$ is an augmented algebra object in the
tensor category of right $H$-modules. 
One forms the smash product algebra $ H\# A$
which is the vector space $ H\otimes A$ 
with multiplication given by
\[ (h'\#a)(h\# b)=\sum_{(h)}h'h_{(1)}\#(a\cdot h_{(2)})b.\]  

 For a left $H$-module augmented algebra $A$,
 one can similarly define $ A\# H$. If $A$ is
 a Hopf algebra with the comultiplication  $A\to A\otimes A$ being 
 an $H$-module homomorphism and $H$ is
 cocommutative, then $H\# A$ is a Hopf algebra
 containing both $A$ and $H$ as Hopf
 subalgebras with $A$ being normal. 
For a general augmented algebra $A$, $H\# A$
is an augmented algebra and $A$ is a normal
subalgebra \cite[2.1]{ABG} such that the
quotient 
$(H\# A)/\!\!/A:=(H\# A)/\langle A^+\rangle
\cong H$ as augmented algebras.  Note that 
$ (H\# A)/\!\!/A\cong (H\# A)\otimes_A \bfk$
as 
augmented algebras.

A left $A$-module $X$ is said to have a
compatible left $H$-module action if it
extends to a left $H\# A$-module. However, 
the condition is  slightly different from the
right module structure 
\[ a(h(x))=\sum_{(h)}h_{(1)}((a\cdot h_{(2)})x).\]
This is equivalent to $M$ being a left 
$ H\# A$-module. Similarly, one can define 
a right $A$-module $X$ with a compatible 
right $H$-module action. This is equivalent 
to $X$ being a right $ H\#A$-module.  
One notes that   right $H\#A$-modules are 
just  the right module objects for the 
algebra object $A$ in the tensor category of
right $H$-modules. One can also define a left
$A$ module $X$ with a compatible right
$H$-module structure. This is equivalent to
$X$ being a left module object  of the algebra
object $A$ in the tensor category of right
$H$-modules.

If $A$ and $A'$ are two right $H$-module algebras, we say that   an $A'$-$A$-bimodule $X$ has a compatible right $H$-module structure if 
\[ (a'xa)\cdot h=\sum_{(h)}(a'\cdot h_{(1)})(x\cdot h_{(2)})(a\cdot h_{(3)}).
\]
This is equivalent to $X$ being a bimodule object of the two algebra objects $A'$ and $A$ in the tensor category of $ H$-modules.
In case $A$ has  a trivial right $H$-module structure, $H\# A=H\otimes A$ is simply the tensor product algebra. On any left (or right) $H\# A$-module $X$, the actions of $ H$ and $A$ commute.  
 
Let $X$ be a left $A$-module with a compatible right $H$-module structure, i.e., the structure map $A\otimes X\to X$ is a homomorphism of $H$-modules. This means that $X$ is a left $A$-module in the tensor category of right $H$-modules. If $Y $ is an $A'$-$A$-bimodule with a compatible right $H$-module structure (i.e., $A'\otimes Y\otimes A \to Y$ is a homomorphism of right $H$-modules,  then  there is a right $H$-module structure on $Y\otimes_{A}X$ given via the tensor product of right $H$-modules.
This setting is equivalent to defining the tensor product over an algebra object $A$ of a left $A$-module object $X$ and a right $A$-module $Y$ in the tensor category of right $H$-modules. In particular, $Y\otimes_A X$ also has a left $A'$-module structure compatible with the right $H$-module structure.

We now consider the case $B=H\# A$. Then 
$Y=B$ is a $B$-$A$-bimodule with left and
right multiplications in $B$ with $A$ being 
a subalgebra of $B$.
The Hopf algebra $H$ also acts on $Y$ by right multiplication. Thus, 

\begin{eqnarray*} ((h'\# a')a)h&=& (h'\# (a'a))h \\
&=&\sum_{(h)}h'h_{(1)}\# ((a'a)\cdot h_{(2)}) \\
&=&\sum_{(h)}h'h_{(1)}\# ((a'\cdot h_{(2)})(a\cdot h_{(3)}) \\ 
&=&\sum_{(h)}(h'\#a')h_{(1)})(1\# (a\cdot h_{(2)})) \\
&=&\sum_{(h)}(h'\#a')h_{(1)}) (a\cdot h_{(2)}).
\end{eqnarray*}

We now consider  the standard Hochschild complex $S_n(A)=A\otimes A^{\otimes n}\otimes A $ of $A$-$A$-bimodules \cite[IX.6]{CE}. 
For any left $A$-module $N$, the complex of left $ A$-modules $C_\bullet (N)=S_\bullet(A)\otimes_{A}N$ is an $A$-projective resolution of the left $A$-module $N$ (using \cite[IX.6 (2)]{CE}).

The Hopf algebra $H$ acts on $S_n(A)$ from the right  as the tensor product of right $H$-modules making $A\otimes A^{\otimes n}\otimes A $ compatible with the $A$-$A$-bimodule structure. 
 By writing down the differentials 
 $d_n: S_n(A) \rightarrow S_{n-1}(A) $ 
 in the standard Hochschild complex,  
 one notices that  each $h\in H$ commutes 
 with the differential $d_n$.  
 By applying the functor $-\otimes_A \bfk$,
 we get an $A$-projective resolution  
 $C_\bullet(\bfk)=A\otimes A^{\otimes \bullet}
 \rightarrow \bfk$ of the trivial $A$-module
 $\bfk$ (assuming $A$ is an augmented algebra).
 The $H$-module structure on $ C_n(\bfk)$ is still given by the tensor product structure.  

 Applying the exact functor $B\otimes_A-$  to the complex $C_\bullet(\bfk)$, one  gets  a complex of $B$-modules 
$D_\bullet= B\otimes_{A}C_\bullet(\bfk)=B\otimes A^{\otimes n}$ which is a $B$-projective resolution of $ B\otimes_A \bfk$. 

The right $H$-module  structure on $ C_n(\bfk)$ does not commute with the left $A$-action. With $H$ acting on $B$ via right multiplication (identifying  $H$ with the subalgebra $H\# 1$ of $B$), the tensor product $H$-module structure on $ B\otimes C_n(\bfk)$ is $A$-balanced, i.e.,
the quotient map $ \pi: B \otimes C_n(\bfk)\rightarrow B\otimes_A C_n(\bfk)$ satisfy the following 
\begin{eqnarray*} \pi((b\otimes a c)\cdot h)&=&\pi(\sum_{(h)}(bh_{(1)})\otimes ((ac)\cdot h)) \\
&=&\pi(\sum_{(h)}(bh_{(1)})\otimes ((a\cdot h_{(2)})(c\cdot h_{(3)})))\\
&=&\pi(\sum_{(h)}(bh_{(1)}(a\cdot h_{(2)}))\otimes (c\cdot h)) \\
&=&\pi(\sum_{(h)}((ba)\cdot h_{(1)}))\otimes (c\cdot h_{(2)}))\\
&=&\pi(((ba)\otimes c)\cdot h).
\end{eqnarray*}
Hence,  the $H$-right module structure on the tensor product $B\otimes_{A}C_n(\bfk)$ is well defined and commutes with the left $B$-module structure via left multiplication. Then $H$ acts on each term of $D_n$ from right commuting with the left $B$-module structure and compatible with the right $A$-module structure. Moreover, the action of $h \in H$  commutes with the differentials, and one has a $B$-projective resolution 
\[D_{\bullet} \rightarrow B\otimes _A \bfk.\]
With $k$ being a trivial $H$-module, $H$ acts on each term of the complex and commutes with the differentials and with the left $B$-module structure. Note that the right action of $H$ on $B\otimes_A \bfk$ is given by 
\[(b\otimes_{A} 1)\cdot h = \sum_{(h)} (b\cdot h_{(1)})\otimes_{A} (1\cdot h_{(2)})=\sum_{(h)}(bh_{(1)}\otimes_{A} (\epsilon( h_{(2)})1)\\
=bh\otimes_A 1\]
such that each $h\in H$ acts on the $B$-projective resolution $ D_\bullet \rightarrow B\otimes_A \bfk$  as chain map of left $B$-modules.  In particular, for any  left $B$-module $M$, each $h \in H$ acts on the complex $ \on{Hom}_B(D_n, M)$ from left given by 
\[ (h\phi)(x)=\phi(x\cdot h), \quad  \forall h \in H ,\ x \in D_n.\]
This defines a left action of $H$ on $\Ext^n_B(B\otimes_A \bfk, M)$ making it a left $H$-module. We claim that this action of $H$ is exactly the left action of $ B/\!\!/A$-action on $ \Ext^n_B(B/\!\!/A, M)$ as described in \cite[XVI Thm 6.1]{CE}. 
Each $h\in H$ defines a left $B$-module homomorphism $B/\!\!/A\rightarrow B/\!\!/A$ by 
$ (b\otimes_A 1)\mapsto (b\otimes_A 1)\cdot h=bh\otimes_A 1$. For any $B$-projective resolution $P_\bullet \rightarrow B\otimes_A \bfk$, there is a lift $\tilde{h}_\bullet: P_\bullet\rightarrow P_\bullet$ as a chain map of $B$-module, which defines the map $ h: \Ext_B^\bullet(B/\!\!/A, M)\rightarrow \Ext_B^\bullet(B/\!\!/A, M)$, which is independent of the choice of the projective resolution and the lift $ \tilde{h}_\bullet $ by using the Comparison Theorem for Resolutions \cite{W}. The $h$ action on the complex $ B\otimes_A C_\bullet(\bfk)$ from the right $H$-module is such a lift and thus defines the same action on $ \Ext_B^{\bullet}(B/\!\!/A, M)$.

 On the other hand,  we have the natural isomorphism of $H$-module homomorphism
\[ \hom_{B}(B\otimes_A C_n(\bfk), M)\stackrel{\rho}{\cong}  \hom_{A}( C_n(\bfk), M)=\hom_{k}(A^{\otimes n}, M)\]
such that 
\begin{eqnarray*} 
\rho(h\phi)(x)&=&\phi((1\otimes_A x)h)\\
&=&\sum_{(h)}\phi(h_{(1)}\otimes_A (x\cdot h_{(2)}))\\
&=&\sum_{(h)}h_{(1)}\phi(1\otimes_{A} (x\cdot h_{(2)})) \\
&=&\sum_{(h)}h_{(1)}(\rho(\phi)(x\cdot h_{(2)})).
\end{eqnarray*}
The right hand side is exactly the $H$-action described in \cite[2.8]{BNPP} on the reduced bar resolution. This verifies the claim in  \cite{BNPP} that the $H$-action using the reduced bar resolution on $\Ext_{B}^\bullet (B/\!\!/A, M)$ is the same as $H=B/\!\!/A$ action described in \cite[XVI Thm 6.1]{CE}.

The advantage of using $H$-action on the bar resolution (or standard) resolution of $A$ is that the Hopf algebra action of $H$ respects the cup product in cohomology.  We summarize our observations as follows. 

\begin{theorem} Let $H$ be a Hopf algebra and $A$ be an augmented algebra which is a right $H$-module algebra over a field $ \bfk$. Then for any left $A$-module $M$ with a compatible left $H$-module structure  there is a natural (functorial in $M$) left $H$ action  on 
$\on{H}^*(A, M)$. Furthermore, under the conditions that $A$ is a Hopf algebra which is also a right $H$-module algebra such that the co-multiplication of $A$ is also a homomorphism of $H$-modules and  $M$ is a left $A$-module algebra extending to a left $H\#A$-module. Then under the cup product \cite[XI.4]{CE}, 
$\on{H}^*(A, M)$ a graded left $ H$-module algebra.    
\end{theorem} 

\subsection{Actions by subalgebras of Hopf algebras on cohomology} \label{sec:2.6} We now apply the above Hopf algebra actions on cohomology groups to cases which are not Hopf algebras. Let $P$ be a subalgebra of a Hopf algebra $H$. Then $P$ is an augmented algebra and the restriction to $P$ of the $H$-action on $A$, makes $A$ into a $P$-module (although one cannot technically call it a $P$-module algebra because $P$ is not necessarily a Hopf algebra). Then $P$  also acts on the cohomology $\Ext_{A}^n(k, M) $ for any left $A$-module $M$ with a compatible left $H$-module structure. 

 Note that $H$ is a right $H$-module algebra under the right adjoint action defined by 
\[ h'\cdot h=\sum_{(h)}S(h_{(1)})h'h_{(2)}.\]
In this case, we have \cite[Prop. 4.2]{Lin2}
\begin{equation}\label{rightadjoint}
 h'h=\sum_{(h)} h_{(1)}(h'\cdot h_{(2)}). 
 \end{equation}
  If $A$ is a subalgebra of $H$ and is $H$-stable under the right adjoint action, then $A$ is a normal subalgebra of $H$ as an augmented algebra.
 If $M$ is a left $H$-module, then the restriction to $A$ makes $M$ a left $A$-module which is compatible with $H$-action since
\[ a(h(m))=\sum_{(h)} h_{(1)}((a\cdot h_{(2)})m)\]
by using \eqref{rightadjoint}.  In this case, we can say more about the $H$-action on $ \on{H}^n(A, M)$. In the construction of the $H$-action on $\on{H}^n(A, M)$,
instead of forming the algebra $B=H\# A$ which includes $A$ as a subalgebra and 
applying the functor $ B\otimes_A-$ to the complex $C_\bullet(\bfk)$, we simply replace $B$ by $H$ and 
apply the functor $ H\otimes _A-$ again with the right $H$-module structure on $H$ by right 
multiplication. Assuming that $ H$ is flat as a right $A$-module, every computation in Section~
\ref{sec:2.5} still applies. In this case, $H\otimes \bfk=H/\!\!/A$. 
Thus  we get an $H$-action on $ \Ext_A^n(\bfk, M)=\Ext _H^n(H\otimes_{A}\bfk, M)$ which is the same 
as the $H$-action described in Section~\ref{sec:2.5} via the chain map 
$\phi\otimes_A 1: (H\#A)\otimes_A C_n(\bfk)\rightarrow H \otimes_A C_n(\bfk)$ with $\phi: H\# A\rightarrow H$ 
defined by $\phi(h\# a)=ha \in H$. Note that $ \phi $ is a homomorphism of $\bfk$-algebras, in 
particular, it is a right $A$-module homomorphism.  On the other hand, the $H$-action on
 $ \Ext_H^n(H/\!\!/A, M)$ defined in \cite{CE} is obtained by lifting the right multiplication of $H$ on $H/\!\!/A$ to 
 an $H$-projective resolution of $H/\!\!/A$. In this case, we consider the standard complex 
 $S_\bullet(H/\!\!/A)$ (but with a change of the differentials to $ (-1)^{n}d_n$ and the homotopy map 
 $s'_n: a_0\otimes\cdots \otimes a_{n+1}\mapsto a_0\otimes\cdots \otimes a_{n+1}\otimes 1$. 
 Then consider the complex  $ H\otimes_H S_\bullet(H/\!\!/A)\rightarrow  H\otimes_H H/\!\!/A$, which is 
an $H$-projective resolution of the $H$-module $H/\!\!/A$. The right $H/\!\!/A$-module structure (also an $H$-module structure)  on 
the resolution defines an $H$-action on $ \Ext_H^\bullet (H/\!\!/A, M)$ which identifies with the $H$-action on $\Ext_H^\bullet (H/\!\!/A, M)\cong \Ext_A^\bullet (\bfk, M)$. Hence, the restriction of the $H$-module 
structure on $ \Ext_A^\bullet (\bfk, M)$ to the subalgebra $ A \subseteq H$ is trivial. 

\begin{proposition} \label{triviality} If $A$ is a subalgebra of $H$ which is stable under the right adjoint $H$-action and $ H$ is a right flat $A$-module, then the subalgebra $A$ action on  $\on{H}^\bullet (A, M)$ is trivial when considered as restriction from the action of $H$.
\end{proposition}

If further $A\subseteq B$ are  subalgebras of $H$ and both are invariant under the right adjoint action of $H$, then $ A$ is a normal subalgebra of $B$. In this case, $ B/\!\!/A$ is also a right $H$-module algebra. Then for any left $H$-module $M$, there are 
$H$-actions  on each term of the spectral sequence described in \cite[Lemma 2.8.1]{BNPP} 
\[ E^{i,j}_2=\on{H}^i(B/\!\!/A, H^j(A, M))\Rightarrow \on{H}^{i+j}(B,M),\]
where $H$-modules and the differentials are $H$-module homomorphisms.

If $A$ is a Noetherian central subalgebra of $H$ then the  right $H$-adjoint action on $A$ is trivial. Hence, the action of $H$ on any projective $A$-module $P$ is trivial and on
$\hom_{\bfk}(P, M) $ is given by  $(h\cdot f)(x)=h\cdot (f(x))$ for all $ f \in \hom_{\bfk}(P, M)$ and $x\in P$. In particular, if $M$ is finitely generated and projective over $\bfk$ then 
\begin{equation} \hom_{\bfk}(P, M) \cong M\otimes \hom_{\bfk}(P, \bfk) 
\end{equation}
as $H$-modules. If $M$ is trivial as $A$-module and $P$ is finitely generated and projective over $A$,
then one has an isomorphism of $H$-modules
\begin{equation*} \hom_{A}(P, M) \cong M\otimes \hom_{A}(P, \bfk) 
\end{equation*}
for all $H$-modules which is a direct limit of submodules which are finitely generated and  projective over $\bfk$.
\begin{proposition} \label{tensor-identity} Let $\bfk$ be field. If $A$ is a finitely generated central subalgebra  of a Hopf algebra $H$ and $M$ is a locally finite left $H$-module  such that its restriction to $A$ is trivial, then one has an $H$-module isomorphism
 \begin{equation} \Ext_A^n(\bfk, M) \cong M\otimes \Ext_A^n(\bfk, \bfk) .
\end{equation}
\end{proposition}

Let $N\subseteq A \subseteq B \subseteq H$ be subalgebras which are stable under the right adjoint action of $ H$ on $H$. Then $H$ acts on $ A'=A/\!\!/N \subseteq B'=B/\!\!/N$.  Then $H$ acts on each of the term of the spectral sequence 
\[E^{i,j}_2= \on{H}^i(B'/\!\!/A', \on{H}^j(A', M))\Rightarrow H^{i+j}(B', M)\]
for each $H$-module $M$ whose restriction to $N$ is trivial. 
If $A$ is a central subalgebra of $H$, then $H$ acts trivially on both $N$ and $A$, and $H/\!\!/N$ also acts on $A'=A/\!\!/N$ trivially. Furthermore, if $M$ is $ A$-trivial   then 
 \begin{equation} \Ext_{A'}^n(\bfk, M) \cong M\otimes \Ext_{A'}^n(\bfk, \bfk) 
\end{equation}
is an isomorphism of $H/\!\!/N$-module. In particular it is an isomorphism of $B'$-module. Therefore, the action of $H/\!\!/N$ on the tensor product module arises from the $H$-action on the tensor product.  
 
 \subsection{Module isomorphisms for parabolic subgroups} \label{parabolic}
 
 Let $ G$ be a connected  algebraic group over a field $\bfk$ of any characteristic. Moreover, let $G\on{-mod}$ be the category of finite dimensional rational $G$-modules.  Let $B\subseteq G$ be a Borel subgroup. Then $ G/B$ is a complete algebraic $\bfk$-variety. Thus the  global sections of the trivial line bundle is $\bfk$. In particular the induced representation $\on{ind}_B^G(\bfk)\cong \bfk$. Let $M$ be any rational $G$-module. Then we have $ \on{ind}_B^G(M)\cong M\otimes \on{ind}_B^G(\bfk)\cong M$.   If $ \mathcal C$ is any category, we use $ \on{Gpd}(\mathcal C)$ to denote the groupoid of $\mathcal C$, which has the same objects but taking isomorphisms only. 
 \begin{proposition} For any connected algebraic $\bfk$-group $G$ and $B$ be a Borel subgroup. If $ M$ and $ M'$ are two rational $G$-modules, then $M\cong M'$ in the category of $G$-modules if and only if $ \on{res}_B^G(M)\cong \on{res}_B^G(M')$ in the category of rational $B$-modules. In particular, the restriction functor $ \on{res}_B^G: G\on{-mod}\to B\on{-mod}$, which is not full in general, defines a fully faithful functor of groupoids 
 $\on{Gpd}(G\on{-mod})\to \on{Gpd}(B\on{-mod})$.
 \end{proposition}
 
 \begin{proof}
   If $ \phi: M\to M'$ is a $B$-module isomorphism, then $ \on{Ind}_B^G(\phi): \on{Ind}_B^G(M)\to  \on{Ind}_B^G(M')$ is an isomorphism of $ G$-modules.  Since both $M$ and $M'$ are $G$-modules, the tensor identity implies that $\on{Ind}_B^G(M)\cong \on{Ind}_B^G(\bfk)\otimes M\cong M$ and $\on{Ind}_B^G(M')\cong M'$.  Here we used the tensor identity \cite[I. 3.6]{Jan1} and the fact that $\on{Ind}_B^G(\bfk)=\bfk$. 
\end{proof}
 
 This can be extended to representations of Hopf algebras. 
 Assume that $ H$ is a Hopf algebra over $ \bfk$ and let $ \mathcal C(H)$ be a full tensor subcategory of $ H$-modules which is closed under tensor product (defined by a topology as described in \cite{Lin1}). If $ D$ is a Hopf subalgebra of $H$, consider the full subcategory $ \mathcal C(D)$ of $D$-modules such that the restriction functor 
 $ \on{res}_D^H: \mathcal C(H)\to \mathcal C(D) $ has a left adjoint functor $ \on{ind}_D^H:  \mathcal C(D)\to \mathcal C(H)
$. The argument in \cite{Lin1} implies that tensor identity holds  if $H$ has a bijective antipode. In fact this follows from the adjointness of functors by using the following: 
\vskip .15cm
\noindent
(1) $\on{Hom}_H(M, N)=\on{Hom}_\bfk(M, N)^{H}$;
\vskip .15cm
\noindent
(2) $\on{Hom}_H(M\otimes P^*, N)\cong \on{Hom}_H(M, N\otimes P)$ and $\on{Hom}_H({}^*\!P\otimes M, N)\cong \on{Hom}_H(M, P\otimes N)$ for all $ H$-modules $P$ which are finitely generated and projective over $ \bfk$. Here the $H$-module structures on $ {}^*P=\on{Hom}_{\bfk}(P, \bfk)$ (resp.  $ P^*=\on{Hom}_{\bfk}(P, \bfk)$)  are defined by $ (hf)(p)=f(S(h)p)$ (resp.  $ (hf)(p)=f(S^{-1}(h)p)$), where $M, N$ can be any $H$-modules. 
\vskip .15cm
\noindent
(3) $\on{ind}_D^H$ commutes with direct limits and  every object in the category $ \mathcal C(H)$  is a direct limit of objects that are finitely generated over $ \bfk$. The tensor identity can be stated in the following two ways: for any $ D$-module $N$ and $ H$-module $V$ in $ \mathcal C(H)$, 
\[ \on{ind}_D^H(N\otimes V)\cong \on{ind}_D^H(N)\otimes V\quad \text{ and } \quad \on{ind}_D^H(V\otimes N)\cong V\otimes \on{ind}_D^H(N).\]
This is a generalization of  \cite[I. 6.13]{Jan1}  from groups to Hopf algebras. The proof of  \cite[Prop. 3.7]{Lin1} reduces to continuous dual Hopf algebra and realizes $\on{ind}_D^H$ as cotensor product of comodules and uses \cite{PW}.

\begin{proposition} \label{module-isomorphism} Let $ \bfk$ be a field.
 Assume that $H$ is a Hopf algebra with bijective antipode and the category $ \mathcal C(H)$ has the property that  every object is locally finite. If $D$ is a Hopf subalgebra of $H$ such that $\on{ind}_D^H(\bfk)\cong \bfk$, then any two modules $V_{1}, V_{2},\in \mathcal C(H)$ are isomorphic if and only if $ \on{res}_D^H(V_{1})\cong  \on{res}_D^H(V_{2})$ as $ D$-modules. 
\end{proposition} 

\begin{proof} Clearly, if $V_{1}$ is isomorphic to $V_{2}$ over $H$ then these modules must be isomorphic over $D$. Suppose that $V_{1}$ and $V_{2}$ are isomorphic when restricted to $D$. Then by the tensor identity and the fact that $\on{ind}_D^H(\bfk)\cong \bfk$, one has that these modules must be isomorphic over $H$. 
\end{proof}

The proposition  holds for various Hopf algebras such as Lusztig's divided power quantum groups and their Borel subalgebras. The tensor identity holds and the induced representation of the trivial module from a Borel subalgebra remains trivial (cf. \cite{APW}).



\end{document}